\documentclass[12pt]{article}
\usepackage{graphicx}
\usepackage{amscd, amssymb}
\usepackage{enumerate}
\usepackage{hyperref}
\usepackage{lscape}
\newenvironment{eq}{\begin{equation}}{\end{equation}}
\newenvironment{proof}{{\bf Proof}:}{\vskip 5mm }

\newtheorem{proposition}{Proposition}[subsection]
\newtheorem{lemma}[proposition]{Lemma}
\newtheorem{definition}[proposition]{Definition}

\newtheorem{example}[proposition]{Example}
\newtheorem{remark}[proposition]{Remark}

\newtheorem{problem}[proposition]{Problem}
\newtheorem{construction}[proposition]{Construction}
\newcommand{\llabel}[1]{\label{#1}}
\newcommand{\comment}[1]{}
\newcommand{\sr}{\rightarrow}

\newcommand{\dd}{\diamond}

\newcommand{\nn}{{\bf N\rm}}

\newcommand{\wt}{\widetilde}

{\vskip
3mm}

\begin{document}
\parskip = 2mm
\begin{center}
{\bf\Large A C-system defined by a universe category\footnote{\em 2000 Mathematical Subject Classification: 
03F50, 
03B15, 
03G25 
}}

\vspace{3mm}

{\large\bf Vladimir Voevodsky}\footnote{School of Mathematics, Institute for Advanced Study,
Princeton NJ, USA. e-mail: vladimir@ias.edu}$^,$\footnote{Work on this paper was supported by NSF grant 1100938 and Clay Mathematical Institute.}
\vspace {3mm}

{July 2015}  
\end{center}

\begin{abstract}
This is the third paper in a series started in \cite{Csubsystems}.  In it we construct a C-system $CC({\cal C},p)$ starting from a category $\cal C$ together with a morphism $p:\wt{U}\sr U$, a choice of pull-back squares based on $p$ for all morphisms to $U$ and a choice of a final object of $\cal C$.  Such a quadruple is called a universe category. We then define universe category functors and construct homomorphisms of C-systems $CC({\cal C},p)$ defined by universe category functors. 

In the last section we give, for any C-system $CC$, three different constructions of pairs $(({\cal C},p),H)$ where $({\cal C},p)$ is a universe category and $H:CC \sr CC({\cal C},p)$ is an isomorphism. 
\end{abstract}

\subsection{Introduction}

The concept of a C-system in its present form was introduced in \cite{Csubsystems}. The type of the C-systems is constructively equivalent to the type of contextual categories defined by Cartmell in \cite{Cartmell1} and \cite{Cartmell0} but the definition of a C-system is slightly different from the Cartmell's foundational definition.

In \cite{Cofamodule} we constructed for any pair $(R,LM)$ where $R$ is a monad on $Sets$ and $LM$ a left $R$-module with values in $Sets$ a C-system $CC(R,LM)$. In the particular case of pairs $(R,LM)$ corresponding to binding signatures (cf. \cite{Aczel}, \cite{FPT}, \cite[p.228]{HM2007}) the regular sub-quotients of $CC(R,LM)$ are the C-systems corresponding to dependent type theories of the Martin-Lof genus.

In this paper we describe another construction that  generates C-systems. This time the input data is a quadruple that consists of a category $\cal C$, a morphism $p:\wt{U}\sr U$ in this category, a choice of pull-back squares based on $p$ for all morphisms to $U$ and a choice of a final object in $\cal C$.  Such a quadruple is called a universe category. For any universe category we construct a C-system that we denote by $CC({\cal C},p)$.  

We then define the notion of a universe category functor and construct homomorphisms of C-systems of the form $CC({\cal C},p)$ corresponding to universe category functors. For universe category functors satisfying certain conditions these homomorphisms are isomorphisms. In particular, any equivalence $F:{\cal C}\sr {\cal C}'$ together with an isomorphism $F(p)\cong p'$ (in the category of morphsims) defines a universe category functor whose associated homomorphism of C-systems is an isomorphism. This implies the C-systems that correspond to two different choices of final objects and pull-backs for the same $\cal C$ and $p$ are connected by a given isomorphism which justifies our simplified notation $CC({\cal C},p)$. 

To the best of our knowledge it is the only known construction of a C-system from a category level data that transforms equivalences into isomorphisms. Because of this fact we find it important to present both the construction of the C-system and the construction of the homomorphisms defined by universe functors in detail.

Next we explore the question of how to construct, for a given C-system $CC$,  a universe category $({\cal C},p)$ together with an isomorphism $CC\sr CC({\cal C},p)$. It is clear from the functoriality theorem of the previous section that if this problem has a solution then it has many solutions. We construct three such solutions each having certain advantages and disadvantages. 

The set of universe categories in a given Grothendieck universe has a structure of a 2-category suggested by Definition \ref{2015.03.21.def1}. It seems likely that our main construction extends to a construction of a functor from this 2-category to the 1-category of C-systems. We leave the investigations of the properties of this 2-category and of this functor for the future.

To avoid the abuse of language inherent in the use of the Theorem-Proof style of presenting mathematics when dealing with constructions we use the pair of names Problem-Construction for the specification of the goal of a construction and the description of the particular solution.

In the case of a Theorem-Proof pair one usually refers (by name or number) to the statement when using both the statement and the proof. This is acceptable in the case of theorems because the future use of their proofs is such that only the fact that there is a proof but not the particulars of the proof matter. 

In the case of a Problem-Construction pair the content of the construction often matters in the future use. Because of this we often have to refer to the construction and not to the problem and we assign in this paper numbers both to Problems and to the Constructions. 

Following the approach used in \cite{Csubsystems} we write the composition of morphisms in categories in the diagrammatic order, i.e., for $f:X\sr Y$ and $g:Y\sr Z$ their composition is written as $f\circ g$. This makes it much easier to translate between diagrams and equations involving morphisms. 

The methods of this paper are fully constructive and the style we write in is the ``formalization ready'' style where the proofs are spelled out in detail even when the assertion may appear obvious to the practitioners of a particular tradition in mathematics. This particular paper is written with having in mind the possibility of formalization both in the Zermelo-Fraenkel set theory (without the axiom of choice) and its constructive versions and in any type theory including Church's type theory or HOL. 

Following the distinction that becomes essential in the univalent formalization (cf. \cite{RezkCompletion}) we use the word ``category'' in the contexts where the the corresponding object is used in a way that is functorial for equivalences of categories and the word ``precategory'' otherwise. 

The main construction of this paper was introduced in \cite{NTS}. I am grateful to The Centre for Quantum Mathematics and Computation (QMAC) and the Mathematical Institute of the University of Oxford for their hospitality during my work on the previous version of the paper and to the Department of Computer Science and Engineering of the University of Gothenburg and Chalmers University of Technology for its the hospitality during my work on the present version.

\subsection{Construction of $CC({\cal C},p)$.}

\begin{definition}
\llabel{2009.11.1.def1}
Let $\cal C$ be a category. A universe structure on a morphism $p:\wt{U}\sr U$ in $\cal C$ is a mapping that assigns to any morphism $f:X\sr U$ in $\cal C$ a pull-back square 
$$
\begin{CD}
(X;f) @>Q(f)>>  \wt{U}\\
@Vp_{X,f}VV  @VVpV\\
X @>f>> U
\end{CD}
$$
A universe in $\cal C$ is a morphism $p$ together with a universe structure on it. 
\end{definition}
In what follows we will write $(X;f_1,\dots,f_n)$ for $(\dots((X;f_1);f_2)\dots;f_n)$.

\begin{example}\rm
\llabel{2015.04.06.ex1}
Let $G$ be a group. Consider the category $BG$ with one object $pt$ whose monoid of endomorphisms is $G$. Recall that any commutative square where all four arrows are isomorphisms is a pull-back square. Let $p:pt\sr pt$ be the unit object of $G$. Then a universe structure on $p$ can be defined by specifying, for every $g:pt\sr pt$, of the horizontal morphism $Q(g)$ in the corresponding canonical square. There are no restrictions on the choice of $Q(g)$ since for any such choice one can take the vertical morphism to be $Q(g)g^{-1}$ obtaining a pull-back square. Therefore, the set of universe structures on $p$ is $G^G$. The automorphisms of $BG$ are given by $Aut(G)$ (with two automorphisms being isomorphic as functors if they differ by an inner automorphisms of $G$). Therefore,  there are $(G^G)/Aut(G)$ isomorphism classes of categories with universes with the underlying category $BG$ and the underlying universe morphism being $Id:pt\sr pt$. Note that in this case all auto-equivalences of the category are automorphisms and so simply saying that we will consider universes up to an equivalence of the underlying category does not change the answer. To have, as is suggested by category-theoretic intuition, no more than one universe structure on a morphism one needs to consider categories with universes up to equivalences of categories with universes and then one has the obligation to prove that the constructions that are supposed to produce objects such as C-systems map equivalences of categories with universes to isomorphisms. In the case of the main construction of this paper it is achieved in Lemma \ref{2014.09.18.l1}.
\end{example}

For $f:W\sr X$ and $g:W\sr \wt{U}$ we will denote by $f*g$ the unique morphism such that 
$$(f*g)\circ p_{X,F}=f$$
$$(f*g)\circ Q(F)=g$$
\comment{\begin{remark}\rm
\llabel{2015.03.29.rm1}
Note that we make no assumption about $Q(Id_U)$ being equal to $Id_{\wt{U}}$. While we could make such an assumption we are not allowed to make such an assumption since the question of whether or not a given morphism is an identity morphism need not be decidable and therefore we can not ``normalize'' our constructions by doing a ``case'' on whether a morphism is an identity morphism or not. The importance of this observation (in the context of whether a simplex is degenerate or not) was emphasized in \cite{BCH}.
\end{remark}}
For $X'\stackrel{f}{\sr}X\stackrel{F}{\sr}U$ we let $Q(f,F)$ denote the morphism 
$$(p_{X',f\circ F}\circ f)*Q(f\circ F):(X';f\circ F)\sr (X;F)$$
such that in particular
\begin{eq}\llabel{2015.07.11.eq10}
Q(f,F)\circ Q(F)=Q(f\circ F)
\end{eq}
\begin{lemma}
\llabel{2015.04.06.l0}
The square
\begin{eq}\llabel{2015.07.11.eq1}
\begin{CD}
(X';f\circ F) @>Q(f,F)>> (X;F)\\
@Vp_{X',f\circ F} VV @VVp_{X,F}V\\
X' @>f>> X
\end{CD}
\end{eq}
is a pull-back square.
\end{lemma}
\begin{proof}
Consider the diagram
$$
\begin{CD}
(X';f\circ F) @>Q(f,F)>> (X;F) @>Q(F)>> \wt{U}\\
@Vp_{X',f\circ F} VV @VVp_{X,F}V @VVp V\\
X' @>f>> X @>F>> U
\end{CD}
$$
The composition of two squares of this diagram equals the square with the sides $p_{X',f\circ F}$, $f\circ F$, $Q(f\circ F)$ and $p$, which is a pull-back square. The right hand side square in this diagram is a pull-back square. This implies that the left hand side square is a pull-back square.
\end{proof}
\begin{lemma}
\llabel{2015.07.19.l1}
If $f:X'\sr X$ is an isomorphism then $Q(f,F)$ is an isomorphism.
\end{lemma}
\begin{proof}
It follows from Lemma \ref{2015.04.06.l0} by general properties of pull-back squares. 
\end{proof}
\begin{lemma}
\llabel{2015.04.14.l0}
For $f':X''\sr X'$, $f:X'\sr X$ and $F:X\sr U$ one has
$$Q(f',f\circ F)\circ Q(f,F)=Q(f'\circ f,F)$$
\end{lemma}
\begin{proof}
Both sides of the equality are morphisms to $(X;F)$, therefore it is sufficient to verify that
$$Q(f',f\circ F)\circ Q(f,F)\circ Q(F)=Q(f'\circ f,F)\circ Q(F)$$
and
$$Q(f',f\circ F)\circ Q(f,F)\circ p_{X,F}=Q(f'\circ f,F)\circ p_{X,F}$$
For the first one we have
$$Q(f',f\circ F)\circ Q(f,F)\circ Q(F)=Q(f',f\circ F)\circ Q(f\circ F)=Q(f'\circ f\circ F)$$
and
$$Q(f'\circ f,F)\circ Q(F)=Q(f'\circ f\circ F)$$
and for the second one we have
$$Q(f',f\circ F)\circ Q(f,F)\circ p_{X,F}=Q(f',f\circ F)\circ p_{X', f\circ F}\circ f=p_{X'', f'\circ f\circ F}\circ f'\circ f$$
and
$$Q(f'\circ f,F)\circ p_{X,F}=p_{X'', f'\circ f\circ F}\circ f'\circ f.$$
\end{proof}

\begin{definition}
\llabel{2015.03.21.def2}
A universe category is a triple $({\cal C},p,pt)$ where $\cal C$ is a category, $p:\wt{U}\sr U$ is a morphism in $\cal C$ with a universe structure on it and $pt$ is a final object in $\cal C$.
\end{definition}
We will often denote a universe category by a pair $({\cal C},p)$. 

Let $({\cal C},p)$ be a universe category and $X\in {\cal C}$. Define by induction on $n$ pairs $(Ob_n({\cal C},p),int_{n})$ where $Ob_n=Ob_n({\cal C},p)$ are sets and $int_{n}:Ob_n\sr Ob({\cal C})$ are functions, as follows:
\begin{enumerate}
\item $Ob_0=unit$ where $unit$ is the distinguished set with only one point $tt$ and $int_0$ maps this point to $pt$.
\item $Ob_{n+1}=\amalg_{A\in Ob_n}Hom_{\cal C}(int_n(A),U)$ and $int_{n+1}(A,F)=(int_n(A);F)$.
\end{enumerate}
In what follows we will write $int$ instead of $int_{n}$ since $n$ can usually be inferred.

Define for each $n$ the function $ft_{n+1}:Ob_{n+1}\sr Ob_n$ by the formula $ft_{n+1}(A,F)=A$ and define $ft_0$ as the identity function of $Ob_0$.

For each $B=(ft(B),F)\in Ob_{n+1}$ define $p_B:int(B)\sr int(ft(B))$ as $p_{int(ft(B)),F}$. For $B\in Ob_0$ define $p_B$ as $Id_{int(B)}$.

For each $A\in Ob_m$, $B=(ft(B),F)\in Ob_{n+1}$ and $f:int(A)\sr int(ft(B))$ define $f^*(B)\in Ob_{m+1}$ as
\begin{eq}
\llabel{2015.07.11.eq6}
f^*(B)=(A,f\circ F)
\end{eq}
and $q(f,B):int(f^*(B))\sr int(B)$ as
\begin{eq}
\llabel{2015.07.11.eq7}
q(f,B)=Q(f,F)
\end{eq}
Recall that the concept of a C0-system was defined in \cite[Definition 2.1]{Csubsystems}. 
\begin{problem}
\llabel{2014.09.18.prob1a}
For each universe category $({\cal C},p,pt)$ to define a C0-system $CC0({\cal C},p)$.
\end{problem}
\begin{construction}
\llabel{2014.09.18.constr1a}\rm
We set
$$Ob(CC0({\cal C},p))=\amalg_{n\ge 0} Ob_n({\cal C},p)$$
where $Ob_n=Ob_n({\cal C},p)$ are the sets introduced above. Let 
$$int_{Ob}:Ob(CC0({\cal C},p))\sr {\cal C}$$
be the sum of the functions $int_{n}$. Let
$$Mor(CC0({\cal C},p))=\amalg_{\Gamma,\Gamma'\in Ob(CC0({\cal C},p))}Hom_{\cal C}(int_{Ob}(\Gamma),int_{Ob}(\Gamma'))$$
Define the function 
$$int_{Mor}:Mor(CC0({\cal C},p))\sr Mor({\cal C})$$
by the formula
$$int_{Mor}(\Gamma,(\Gamma',a))=a$$
We will often write simply $int$ for $int_{Ob}$ and $int_{Mor}$. 

The identity morphisms and the composition of morphisms are defined as in $\cal C$. The proofs of the axioms of a category are straightforward. 

The definition of the length function is obvious. 

We define $pt$ as the unique element $(0,tt)$ of $Ob(CC0({\cal C},p))$ of length zero. 

The function $ft:Ob(CC0)\sr Ob(CC0)$ is defined as the sum of functions $ft_n$ defined above. 

The  $p$-morphisms $p_{(n,A)}$ are defined such that $int(p_{(n,A)})=p_A$ where $p_A$ where defined above. 

Similarly one defines the morphisms $q(f,(n+1,B))$ such that $int(q(f,(n+1,B)))=q(f,B)$. 

Lemma \ref{2015.04.14.l0} shows that the structure that we have defined satisfies the axioms of a C0-system given in \cite[Definition 2.1]{Csubsystems}. 
\end{construction}
Let us also note the following formulas. For $\Delta=(n+1, (B,F))$ and $\Gamma=(n,B)$ one has
\begin{eq}\llabel{2015.07.19.eq7}
p_{\Delta}=(\Delta,(\Gamma, p_{int(B),F}))
\end{eq}
For $\Gamma'=(m,A)$, $\Gamma=(n,B)$ and $f:\Gamma'\sr \Gamma$ one has
\begin{eq}
\llabel{2015.07.19.eq2}
f^*(n+1,(B,F))=(m+1,(A,int(f)\circ F))
\end{eq}
\begin{eq}
\llabel{2015.07.19.eq3}
q(f,(n+1,(B,F)))=(f^*(\Delta),(\Delta,Q(int(f),F)))
\end{eq}
\begin{lemma}
\llabel{2015.07.13.l1}
The functions $int_{Ob}$ and $int_{Mor}$ defined above form a fully faithful functor from the category underlying the C0-system $CC0({\cal C},p)$ to $\cal C$.
\end{lemma}
\begin{proof}
Easy from the construction.
\end{proof}
\begin{remark}\rm
The image of $int$ on objects consists of those objects for which the unique morphism to $pt$ can be represented as a composition of morphisms of the form $p_{X,F}$. Note that $int$ need not be an injection on the sets of objects. For example, if $\cal C$ is the one point category with its unique structure of a universe category then $Ob(CC({\cal C},p))$ will be isomorphic to the set of natural numbers.
\end{remark}
\begin{problem}
\llabel{2014.09.18.prob1b}
For each universe category $({\cal C},p,pt)$ to define a C-system $CC({\cal C},p)$.
\end{problem}
\begin{construction}
\llabel{2014.09.18.constr1}\rm
We will define $CC({\cal C},p)$ as an extension of $CC0({\cal C},p)$ using \cite[Proposition 2.4]{Csubsystems}. In particular $Ob(CC)=Ob(CC0)$, $Mor(CC)=Mor(CC0)$ and similarly for the length function, $ft$, $p$-morphisms and $q$-morphisms. 

The canonical squares of $CC0({\cal C},p)$ are of the form 
\begin{eq}\llabel{2015.07.09.eq3}
\begin{CD}
f^*(\Gamma) @>q(f,\Gamma)>> \Gamma\\
@Vp_{f^*(\Gamma)}VV @VVp_{\Gamma}V\\
\Gamma' @>f>> ft(\Gamma)
\end{CD}
\end{eq}
For $\Gamma=(n+1,(B,F))$ where $B\in Ob_n({\cal C},p)$ and $F:int(A)\sr U$, $\Gamma'=(m,A)$ where $A\in Ob_m({\cal C},p)$,  and $f=(\Gamma',(\Gamma,a))$ the image of this square under the functor $int$ is of the form
$$
\begin{CD}
(int(A);a\circ F) @>Q(a,F)>> (int(B);F)\\
@Vp_{int(A),a\circ F}VV @Vp_{int(B),F}VV\\
int(A) @>F>> int(B)
\end{CD}
$$
This is one of the squares of the form (\ref{2015.07.11.eq1}) and therefore by Lemma \ref{2015.04.06.l0} it is a pull-back square. Since $int$ is fully faithful by Lemma \ref{2015.07.13.l1}, the squares (\ref{2015.07.09.eq3}) are pull-back squares in the codomain of a fully faithful functor and therefore they are also pull-back squares in the domain of this functor, i.e., in $CC0({\cal C},p)$. In view of \cite[Proposition 2.4]{Csubsystems} this implies that the C0-system $CC0({\cal C},p)$ has a unique structure of a C-system and we denote this C-system by $CC({\cal C},p)$.  
\end{construction}
\begin{remark}
\rm
\llabel{2015.07.11.rem1}
Recall that in \cite{Csubsystems} we suggested the notation $Ob_n(CC)$ for the set of objects of length $n$ of a C-system $CC$. We will avoid using this notation here because the sets $Ob_n({\cal C},p)$ are not equal to the subsets of elements of length $n$ in $CC({\cal C},p)$. Indeed, the elements of $\{\Gamma\in Ob(CC({\cal C},p))\,|\,l(\Gamma)=n\}$ are not the elements of $Ob_n({\cal C},p)$ but pairs of the form $(n,A)$ where $A\in Ob_n({\cal C},p)$.
\end{remark}
\begin{example}\rm
\llabel{2015.06.15.ex1}
An important example of a C-system of the form $CC({\cal C},p)$  is ``the'' C-system $Fam$ of families of sets considered in \cite{Cartmell0} and \cite{Cartmell1}. The definition of $Fam$ in \cite[p.238]{Cartmell1} as well as the preceding it discussion in \cite[p.232]{Cartmell1} is somewhat incomplete in that the notion of ``a set'' and moreover the notion of ``a family of sets'' are taken as being  uniquely determined by some previous agreement that is never explicitly referred to. 

To define $Fam$ as a C-system of the form $CC({\cal C},p)$ let us choose two Grothendieck universes $U$ and $U_1$ in our set theory such that $U_1$ is an element of $U$. One then defines the category $Sets(U)$ of sets as the category whose set of objects is $U$ and such that for $X,Y\in U$ the set of morphisms from $X$ to $Y$ in $Sets(U)$ is the set of functions from $X$ to $Y$ in the ambient set theory (which automatically is an element of $U$). This category will contain $U_1$ as an object and also, because of the closure conditions that $U$ satisfies, it will contain as an object the set $\wt{U}_1$ of pairs $(X,x)$ where $X\in U_1$ and $x\in X$. Since morphisms in $Sets$ are the same as functions in the ambient set theory we also get $p_{U_1}:\wt{U}_1\sr U_1$ that takes $(X,x)$ to $X$. Using the standard construction of pull-backs in sets we obtain a universe structure on $p$. Now we can define:
$$Fam(U,U_1):=CC(Sets(U),p_{U_1})$$
The explicit definition given in \cite{Cartmell1} avoids the use of the second universe (universe $U$ in our notations) by constructing the same C-system ``by hand''. In our approach we have to use $U$ but the resulting C-system does not depend on $U$. Indeed, if our set theory assumes two Grothendieck universes $U$ and $U'$ such that both contain $U_1$ as an element then one can show that
\begin{eq}
\llabel{2015.06.15.eq1}
CC(Sets(U),p_{U_1})=CC(Sets(U'),p_{U_1})
\end{eq}
where the equality means in particular that the sets of objects of these two C-systems are equal as sets. Because of this one can denote this C-system as $Fam(U_1)$.
\end{example}

\subsection{On homomorphisms of C-systems}
We will need below the concept of a homomorphism of C-systems. Homomorphisms of C-systems were defined in \cite[Remark 2.8]{Csubsystems}. Let us recall it here in a more detailed form.
\begin{definition}
\llabel{2015.06.29.def1}
Let $CC_1$, $CC_2$ be C-systems. A homomorphism $F$ from $CC_1$ to $CC_2$ is a pair of functions $F_{Ob}:Ob(CC_1)\sr Ob(CC_2)$, $F_{Mor}:Mor(CC_1)\sr Mor(CC_2)$ such that:
\begin{enumerate}
\item $F$ commutes with the length functions, i.e., for all $X\in Ob(CC_1)$ one has
$$l(F_{Ob}(X))=l(X)$$
\item $F$ commutes with the $ft$ function, i.e., for all $X\in Ob(CC_1)$ one has
$$ft(F_{Ob}(X))=F_{Ob}(ft(X))$$
\item $F$ is a functor, i.e., one has:
\begin{enumerate}
\item $F_{Mor}$ and $F_{Ob}$ commute with the domain and codomain functions,
\item for all $X\in Ob(CC_1)$ one has 
$$F_{Mor}(Id_X)=Id_{F_{Mor}(X)}$$
\item for all $f,g\in Mor(CC_1)$ of the form $f:X\sr Y$, $g:Y\sr Z$ one has
$$F_{Mor}(f\circ g)=F_{Mor}(f)\circ F_{Mor}(g)$$
\end{enumerate}
\item $F$ takes canonical projections to canonical projections, i.e., for all $X\in Ob(CC_1)$ one has
$$p_{F_{Ob}(X)}=F_{Mor}(p_X)$$
\item $F$ takes $q$-morphisms to $q$-morphisms, i.e., for all $X,Y\in Ob(CC_1)$ such that $l(Y)>0$ and all $f:X\sr ft(Y)$ one has
$$F_{Mor}(q(f,Y))=q(F_{Mor}(f),F_{Ob}(Y))$$
\item $F$ takes $s$-morphisms to $s$-morphisms, i.e., for all $X,Y\in Ob(CC_1)$ such that $l(Y)>0$ and $f:X\sr Y$ one has
$$s_{F_{Mor}(f)}=F_{Mor}(s_f)$$
\end{enumerate}
\end{definition}
In what follows we will write $F$ for both $F_{Ob}$ and $F_{Mor}$ since the choice of which one is meant is determined by the type of the argument. Note that the condition that $F$ commutes with the domain function together with the $q$-morphism condition implies that for all $X,Y\in Ob(CC_1)$ such that $l(Y)>0$ and all $f:X\sr ft(Y)$ one has
\begin{eq}
\llabel{2015.06.29.eq1}
F(f^*(Y))=F(f)^*(F(Y))
\end{eq}
\begin{lemma}
\llabel{2015.06.29.l2}
Let $F:CC_1\sr CC_2$ and $G:CC_2\sr CC_3$ be homomorphisms of C-systems. Then the compositions of functions $F_{Ob}\circ G_{Ob}$ and $F_{Mor}\circ G_{Mor}$ is a homomorphism of C-systems.
\end{lemma}
\begin{proof}
The proof is relatively straightforward but long and we leave it for the formal version(s) of the paper.
\end{proof}
\begin{remark}\rm
\llabel{2015.06.29.rem1}
Since homomorphisms of C-systems are pairs of functions between sets satisfying certain conditions and the composition is given by composition of these functions, the associativity and unitality of this composition follows easily from the associativity and unitality of the composition of functions between sets. Therefore, if we restrict our attention to the C-systems whose sets $Ob$ and $Mor$ are elements of a chosen set (``universe'') $U$ that contains natural numbers and is closed under the power-set operation, then such C-systems, their homomorphisms, compositions of these homomorphisms and the identity homomorphisms form a category of C-systems in $U$. 
\end{remark} 
\begin{lemma}
\llabel{2015.06.29.l1}
Let $CC_1$, $CC_2$, $F_{ob}$ and $F_{Mor}$ be as above. Assume further that these data satisfies all of the conditions of the definition except, possibly, the $s$-morphisms condition. Then it satisfies the $s$-morphisms condition and forms a homomorphism of C-systems.
\end{lemma}
\begin{proof}
Let $f:X\sr Y$ be as in the $s$-morphism condition. We need to show that $F(s_f)=s_{F(f)}$. Observe first that the right hand side is well defined since $l(F(Y))=l(Y)>0$. We have $F(s_f):F(X)\sr F((f\circ p_Y)^*(Y))$ and $s_{F(f)}:F(X)\sr (F(f)\circ p_{F(Y)})^*(F(Y))$. One proves that codomains of both morphisms are equal using that $F$ is a functor, the $p$-morphisms condition and (\ref{2015.06.29.eq1}). 

Since the canonical squares of $CC_2$ are pull-back squares the object $(F(f)\circ p_{F(Y)})^*(F(Y))$ is a fiber product with the projections $q(F(f)\circ p_{F(Y)},F(Y))$ and $p_{F(f)^*(F(Y)}$. Therefore it is sufficient to check that one has
\begin{eq}\llabel{2015.06.29.eq2}
F(s_f)\circ q(F(f)\circ p_{F(Y)},F(Y))=s_{F(f)}\circ q(F(f)\circ p_{F(Y)},F(Y))
\end{eq}
and
\begin{eq}\llabel{2015.06.29.eq3}
F(s_f)\circ p_{F(f)^*(F(Y)}=s_{F(f)}\circ p_{F(f)^*(F(Y)}
\end{eq}
We have
$$F(s_f)\circ q(F(f)\circ p_{F(Y)},F(Y))=F(s_f)\circ q(F(f)\circ F(p_Y),F(Y))=$$$$F(s_f)\circ q(F(f\circ p_Y),F(Y))=F(s_f)\circ F(q(f\circ p_Y,Y))=F(s_f\circ q(f\circ p_Y,Y))=F(f)$$
where the first equality holds by condition (4) of Definition \ref{2015.06.29.def1}, the second and the fourth equalities by condition (3), the third equality by condition (5) and the fifth equality by axiom \cite[Definition 2.3(3)]{Csubsystems} of the operation $s$ for $CC_1$. 

On the other hand
$$s_{F(f)}\circ q(F(f)\circ p_{F(Y)},F(Y))=F(f)$$
directly by the axiom \cite[Definition 2.3(3)]{Csubsystems} of the operation $s$ for $CC_2$. This proves (\ref{2015.06.29.eq2}). 

For the equation (\ref{2015.06.29.eq3}) we have
$$F(s_f)\circ p_{F(f)^*(F(Y)}=F(s_f)\circ p_{F(f^*(Y))}=F(s_f)\circ F(p_{f^*(Y)})=F(s_f\circ p_{f^*(Y)})=$$$$F(Id_{X})=Id_{F(X)}$$
where the first equation holds by (\ref{2015.06.29.eq1}), the second one by condition (4), the third one by condition (3), the fourth one by the axiom \cite[Definition 2.3(2)]{Csubsystems} of the operation $s$ for $CC_1$ and the fifth one by condition (3).

On the other hand 
$$s_{F(f)}\circ p_{F(f)^*(F(Y)}=Id_{F(X)}$$
directly by the axiom \cite[Definition 2.3(2)]{Csubsystems} of the operation $s$ for $CC_2$.
This completes the proof of Lemma \ref{2015.06.29.l1}. 
\end{proof}
\begin{remark}\rm
\llabel{2015.o6.29.rem2}
As defined in \cite{Csubsystems}, a C-system without operation $s$ is called a C0-system. The pairs $F_{Ob},F_{Mor}$ that satisfy all of the conditions of Definition \ref{2015.06.29.def1} other than, possibly, the $s$-morphism condition are homomorphisms of C0-systems. Therefore, if one defines a categories of C-systems and C0-systems based on a particular universe of sets as outlined in Remark \ref{2015.06.29.rem1} then Lemma \ref{2015.06.29.l1} implies that the forgetting functor from the category of C-systems in $U$ to C0-systems in $U$ is a full embedding.
\end{remark}

\subsection{Functoriality of $CC({\cal C},p)$}

\begin{definition}
\llabel{2015.03.21.def1}
Let $({\cal C},p,pt)$ and $({\cal C}',p',pt')$ be universe categories. A functor of universe categories from $({\cal C},p,pt)$ to $({\cal C}',p',pt')$ is a triple $(\Phi,\phi,\wt{\phi})$ where $\Phi:{\cal C}\sr {\cal C}'$ is a functor and $\phi:\Phi(U)\sr U'$, $\wt{\phi}:\Phi(\wt{U})\sr \wt{U}'$ are morphisms such that:
\begin{enumerate}
\item $\Phi$ takes the canonical pull-back squares based on $p$ to pull-back squares,
\item $\Phi$ takes $pt$ to a final object of ${\cal C}'$,
\item the square
$$
\begin{CD}
\Phi(\wt{U}) @>\wt{\phi}>> \wt{U}'\\
@V\Phi(p)VV @VVp'V\\
\Phi(U) @>\phi>> U'
\end{CD}
$$
is a pull-back square.
\end{enumerate}
\end{definition}

Let 
$$(\Phi,\phi,\wt{\phi}):({\cal C},p,pt)\sr ({\cal C}',p',pt')$$
be a functor of universes categories. Let $Ob_n=Ob_n({\cal C},p)$ and $Ob'_n=Ob_n({\cal C}',p')$.  Let $int$ and $int'$ be the corresponding functions to $\cal C$ and $\cal C'$. 

Denote by $\psi$ the isomorphism $\psi:pt'\sr \Phi(pt)$.   Define, by induction on $n$, pairs $(H_n, \psi_n)$ where $H_n:Ob_n\sr Ob'_n$ and $\psi_n$ is a family of isomorphisms of the form %
$$\psi_n(A):int'(H_n(A))\sr \Phi(int(A))$$
given for all $A\in Ob_n$. We set:
\begin{enumerate}
\item for $n=0$, $H_0$ is the unique map from a one point set to a one point set and $\psi_0(A)=\psi$,
\item for the successor of $n$ we set
\begin{eq}\llabel{2015.07.11.eq8}
H_{n+1}(A,F)=(H_n(A),\psi_n(A)\circ \Phi(F)\circ \phi)
\end{eq}
and define 
$$\psi_{n+1}(A,F):(int(H_n(A));\psi_n(A)\circ \Phi(F)\circ \phi)\sr \Phi(int(A,F))$$
\end{enumerate}
as the unique morphism such that the left hand side square of the diagram
\begin{eq}\llabel{2009.10.26.eq2}\label{2015.07.09.eq4}
\begin{CD}
int'(H_{n+1}(A,F)) @>\psi_{n+1}(A,F)>> \Phi(int(A,F)) @>\Phi(Q(F))>> \Phi(\wt{U}) @>\wt{\phi}>>\wt{U}'\\
@Vp_{H_{n+1}(A,F)}VV @VV\Phi(p_{(A,F)})V @VV\Phi(p)V @VVp'V\\
int'(H_n(A)) @>\psi_n(A)>> \Phi(int(A)) @>\Phi(F)>> \Phi(U) @>\phi>> U'
\end{CD}
\end{eq}
commutes, i.e.,
\begin{eq}\llabel{2015.07.11.eq15}
\psi(A,F)\circ \Phi(p_{(A,F)})=p_{H(A,F)}\circ \psi(A)
\end{eq}
and 
\begin{eq}\llabel{2015.07.11.eq9}
\psi_{n+1}(A,F)\circ \Phi(Q(F))\circ\wt{\phi}=Q(\psi_n(A)\circ \Phi(F)\circ \phi)
\end{eq}
Note that the existence and uniqueness of $\psi_{n+1}(A,F)$ follows from the fact that the right hand side squares of (\ref{2009.10.26.eq2}) are pull-back squares as a corollary of the definition of a universe category functor and the fact that the canonical square for the morphism $\psi_n(A)\circ \Phi(F)\circ \phi$ commutes. 

Moreover since the outer square of (\ref{2009.10.26.eq2}) is a pull-back square, the left-most square commutes and the two right hand side squares are pull-back squares we conclude that the left hand side square is a pull-back square. In combination with the inductive assumption that $\psi_n(A)$ is an isomorphism this implies that $\psi_{n+1}(A,F)$ is an isomorphism.

In what follows we will write $\psi(A)$ instead of $\psi_n(A)$ since $n$ can often be inferred. 
\begin{lemma}
\llabel{2015.07.11.l1}
The functions $H$ commute with the functions $ft$, i.e., for $A\in Ob_n$ one has
$$ft(H(A))=H(ft(A))$$
\end{lemma}
\begin{proof}
Immediate from the construction.
\end{proof}
Let $A\in Ob_{m}$, $A'\in Ob_{m'}$ and $a:int(A)\sr int(A')$. Define a morphism 
$$H(a):int'(H(A))\sr int'(H(A'))$$
as 
\begin{eq}
\llabel{2015.07.11.eq5}
H(a)=\psi(A)\circ \Phi(a)\circ \psi(A')^{-1}
\end{eq}
\begin{lemma}
\llabel{2015.07.11.l2}
For $\bf\Phi$ as above one has:
\begin{enumerate}
\item for $A\in Ob_n$ one has $H(Id_{int(A)})=Id_{int(H(A))}$,
\item for $a':int(A'')\sr int(A')$ and $a:int(A')\sr int(A)$ one has $H(a'\circ a)=H(a')\circ H(a)$.
\end{enumerate} 
\end{lemma}
\begin{proof}
Immediate from the construction.
\end{proof}
\begin{lemma}
\llabel{2015.07.11.l3}
For $A\in Ob_n$ one has $H(p_A)=p_{H(A)}$.
\end{lemma}
\begin{proof}
If $n=0$ the statement is obvious. For $(A,F)\in Ob_{n+1}$ we have
$$H(p_{(A,F)})=\psi(A,F)\circ \Phi(p_{(A,F)})\circ \psi(A)^{-1}$$
Therefore we need to show that 
$$\psi(A,F)\circ \Phi(p_{(A,F)})=p_{H(A,F)}\circ \psi(A)$$
which is (\ref{2015.07.11.eq15}).
\end{proof}
\begin{lemma}
\llabel{2015.07.11.l4}
Let $A\in Ob_m$, $B=(ft(B),F)\in Ob_{n+1}$ and $a:int(A)\sr int(ft(B))$ is a morphism. Then one has
\begin{eq}\llabel{2015.07.11.eq1b}
H(a^*(B))=H(a)^*(H(B))
\end{eq}
and
\begin{eq}\llabel{2015.07.11.eq2}
H(q(a,B))=q(H(a),H(B))
\end{eq}
\end{lemma}
\begin{proof}
We have
$$H(a^*(ft(B),F))=H(A,a\circ F)=(H(A),\psi(A)\circ\Phi(a\circ F)\circ\phi)$$
and
$$H(a)^*(H(ft(B),F))=H(a)^*(H(ft(B)),\psi(ft(B))\circ\Phi(F)\circ\phi)=$$$$(H(A),H(a)\circ \psi(ft(B))\circ\Phi(F)\circ\phi)$$
Therefore we need to check that
$$\psi(A)\circ \Phi(a)=H(a)\circ \psi(ft(B))$$
which follows from the definition of $H(a)$. 

To prove (\ref{2015.07.11.eq2}) it is sufficient, since $\psi(B)$ is an isomorphism, to show that
$$H(q(a,B))\circ \psi(B)=q(H(a),H(B))\circ \psi(B)$$
In view of (\ref{2015.07.11.eq1b}) both sides are morphisms from $int(H(a^*(B)))$ to 
$$\Phi(int(B))=\Phi((int(ft(B));F))$$
Since the two right squares of (\ref{2009.10.26.eq2}) for $(ft(B),F)$ are pull-back, $\Phi((int(ft(B));F))$ is a fiber product with projections $\Phi(Q(F))\circ\wt{\phi}$ and $\Phi(p_{B})$. Therefore it is sufficient to check two equalities
\begin{eq}\llabel{2015.07.11.eq3}
H(q(a,B))\circ \psi(B)\circ \Phi(Q(F))\circ\wt{\phi}=q(H(a),H(B))\circ \psi(B)\circ \Phi(Q(F))\circ\wt{\phi}
\end{eq}
and
\begin{eq}\llabel{2015.07.11.eq4}
H(q(a,B))\circ \psi(B)\circ \Phi(p_{B})=q(H(a),H(B))\circ \psi(B)\circ \Phi(p_{B})
\end{eq}
Note first that
\begin{eq}\llabel{2015.07.11.eq12}
H(q(a,B))\circ \psi(B)=\psi(a^*(B))\circ \Phi(q(a,(ft(B),F)))=\psi(A,a\circ F)\circ \Phi(Q(a,F))
\end{eq}
where the first equality is by (\ref{2015.07.11.eq5}) and the second by (\ref{2015.07.11.eq7}), and
$$q(H(a),H(B))\circ \psi(B)=q(H(a), (H(ft(B)),\psi(ft(B))\circ\Phi(F)\circ\phi))\circ \psi(B)=$$
\begin{eq}
\llabel{2015.07.11.eq13}
Q(H(a),\psi(ft(B))\circ\Phi(F)\circ\phi)\circ \psi(ft(B),F)
\end{eq}
where is first equality is by (\ref{2015.07.11.eq8}) and the second by (\ref{2015.07.11.eq7}).

For (\ref{2015.07.11.eq3}) we have
$$H(q(a,B))\circ \psi(B)\circ \Phi(Q(F))\circ\wt{\phi}=\psi(A,a\circ F)\circ \Phi(Q(a,F))\circ \Phi(Q(F))\circ \wt{\phi}=$$$$\psi(A,a\circ F)\circ \Phi(Q(a\circ F))\circ \wt{\phi}=Q(\psi(A)\circ \Phi(a\circ F)\circ \phi)$$
where the first equality is by (\ref{2015.07.11.eq12}), second equality is by (\ref{2015.07.11.eq10}) and the third one by (\ref{2015.07.11.eq9}), and
$$q(H(a),H(B))\circ \psi(B)\circ \Phi(Q(F))\circ\wt{\phi}=$$$$Q(H(a),\psi(ft(B))\circ\Phi(F)\circ\phi)\circ \psi(ft(B),F)\circ \Phi(Q(F))\circ\wt{\phi}=$$$$Q(H(a),\psi(ft(B))\circ\Phi(F)\circ\phi)\circ Q(\psi(ft(B))\circ \Phi(F)\circ \phi)=$$$$Q(H(a)\circ \psi(ft(B))\circ\Phi(F)\circ\phi)=$$$$Q(\psi(A)\circ \Phi(a)\circ \psi(ft(B))^{-1}\circ \psi(ft(B))\circ\Phi(F)\circ\phi)=Q(\psi(A)\circ \Phi(a\circ F)\circ \phi)$$
where the first equality is by (\ref{2015.07.11.eq13}), the second equality is by (\ref{2015.07.11.eq9}), the third one by (\ref{2015.07.11.eq10}), and the fourth one by (\ref{2015.07.11.eq5}).
For (\ref{2015.07.11.eq3}) we have:
$$H(q(a,B))\circ \psi(B)\circ \Phi(p_{B})=\psi(A,a\circ F)\circ \Phi(Q(a,F))\circ \Phi(p_{B})=\psi(a^*(B))\circ \Phi(Q(a,F)\circ p_{B})=$$$$\psi(a^*(B))\circ \Phi(q(a,B)\circ p_{B})=\psi(a^*(B))\circ \Phi(p_{a^*(B)})\circ \Phi(a)$$
where the first equality is by (\ref{2015.07.11.eq12}), the second by (\ref{2015.07.11.eq6}) and the assumption that $\Phi$ is a functor, the third one by  (\ref{2015.07.11.eq7}) and the fourth one by the commutativity of the canonical squares and the assumption that $\Phi$ is a functor. 

For the other side we have:
$$q(H(a),H(B))\circ \psi(B)\circ \Phi(p_{B})=q(H(a),H(B))\circ p_{H(B)}\circ \psi(ft(B))=$$$$p_{H(a^*(B))}\circ H(a)\circ \psi(ft(B))=p_{H(a^*(B))}\circ \psi(A)\circ \Phi(a)=\psi(a^*(B))\circ \Phi(p_{a^*(B)})\circ \Phi(a)$$
Where the first equality is by (\ref{2015.07.11.eq15}), the second by the commutativity of the canonical squares, the third by (\ref{2015.07.11.eq5}) and the fourth again by (\ref{2015.07.11.eq15}). This completes the proof of Lemma \ref{2015.07.11.l4}.
\end{proof}
\begin{problem}
\llabel{2014.09.18.prob2}
Let 
$$(\Phi,\phi,\wt{\phi}):({\cal C},p,pt)\sr ({\cal C}',p',pt')$$
be a functor of universes categories. To define a homomorphism $H=H(\Phi,\phi,\wt{\phi})$ from $CC({\cal C},p)$ to $CC({\cal C}',p')$. 
\end{problem}
\begin{construction}
\llabel{2014.09.18.constr2}\rm
We define $H_{Ob}$ as the sum of functions $H_n$ constructed above and for 
$$(\Gamma,(\Gamma',a))\in Mor(CC({\cal C},p))$$
we set
$$H_{Mor}(\Gamma,(\Gamma',a))=(H_{Ob}(\Gamma),(H_{Ob}(\Gamma'),H(a)))$$
where $H(a)$ was constructed above. 

The fact that $H_{Ob}$ commutes with the length functions is immediate from the construction. The fact that it commutes with the $ft$ functions follows from Lemma \ref{2015.07.11.l1}, the fact that $H_{Ob}$ and $H_{Mor}$ form a functor follows from Lemma \ref{2015.07.11.l2}. The fact that $H_{Mor}$ satisfies the $p$-condition follows from Lemma \ref{2015.07.11.l3}. The fact that $H_{Mor}$ satisfies the $q$-condition follows from Lemma \ref{2015.07.11.l4}. 

Applying Lemma \ref{2015.06.29.l1} we conclude that $H=(H_{Ob},H_{Mor})$ is a homomorphism of C-systems.
\end{construction}
\begin{lemma}
\llabel{2014.09.18.l1}
Let $(\Phi,\phi,\wt{\phi})$ be as in Problem \ref{2014.09.18.prob2} and let $H$ be the corresponding solution of Construction \ref{2014.09.18.constr2}. Then one has:
\begin{enumerate}
\item If $\Phi$ is a faithful functor and $\phi$ is a monomorphism then $H$ is an injection of C-systems. 
\item If $\Phi$ is a fully faithful functor and $\phi$ is an isomorphism then $H$ is an isomorphism. 
\end{enumerate}
\end{lemma}
\begin{proof}
Both statements in relation to objects have straightforward proofs by induction on the length. In relation to morphisms the statements follow from the ones about the objects and the fact that $int$ is fully faithful. 
\end{proof}
Lemma \ref{2014.09.18.l1} can be further specialized into the following example.
\begin{example}\rm
\llabel{2015.07.11.ex1}
Let $\cal C$ be a category and $p:\wt{U}\sr U$ a morphism in $\cal C$. Let now $(p_{X,F},Q(F))$ and $(p'_{X,F},Q'(F))$ be two universe structures on $p$ and $pt$ and $pt'$ be two final objects in $\cal C$. These data gives us two universe categories. Let us denote them by ${\cal UC}$ and ${\cal UC}'$. The identity functor $\Phi=Id_{\cal C}$ on $\cal C$ together with the identity morphisms $\phi=Id_{U}$ and $\wt{\phi}=Id_{\wt{U}}$ define a universe category functor ${\bf \Phi}:{\cal UC}\sr {\cal UC}'$. The corresponding homomorphism of C-systems $H_{\bf\Phi}:CC({\cal UC})\sr CC({\cal UC}')$ is an isomorphism with the inverse isomorphism given by the same triple considered as a universe functor from ${\cal UC}$ to ${\cal UC}'$. This example shows that, up to a ``canonical'' isomorphism, the C-system defined by a universe category depends only on the category $\cal C$ and the morphism $p$.
\end{example}

\begin{problem}
\llabel{2015.07.17.prob1}
Let $({\cal C},p)$ be a universe category. Let $CC$ be a C-system. Given the following collection of data:
\begin{enumerate}
\item A functor $I:CC\sr {\cal C}$ from the underlying category of $CC$ to $\cal C$,
\item For each $\Gamma\in CC$ a function
$$u_{\Gamma}:Ob_1(\Gamma)\sr Hom_{\cal C}(I(\Gamma),U)$$
\item For each $\Gamma\in CC$, $\Delta\in Ob_1(\Gamma)$ an isomorphism
$$\gamma_{\Delta}:(I(\Gamma);u(\Delta))\sr I(\Delta)$$
\end{enumerate}
such that
\begin{enumerate}
\item the morphism $\pi_{I(pt)}:I(pt)\sr pt$ is an isomorphism 
\item for each $f:\Gamma'\sr \Gamma$ and $\Delta\in Ob_1(\Gamma)$ one has $u_{\Gamma'}(f^*(\Delta))=I(f)\circ u_{\Gamma}(\Delta)$,
\item for each $f:\Gamma'\sr \Gamma$ and $\Delta\in Ob_1(\Gamma)$ one has $p_{I(\Gamma),u(\Delta)}=\gamma_{\Delta}\circ I(p_{\Delta})$,
\item for each $f:\Gamma'\sr \Gamma$ and $\Delta\in Ob_1(\Gamma)$ one has $\gamma_{f^*(\Delta)}\circ I(q(f,\Delta))=Q(I(f),u(\Delta))\circ \gamma_{\Delta}$
\end{enumerate}
to construct a C-system homomorphism 
$$H(I,u,\gamma):CC\sr CC({\cal C},p)$$
\end{problem}
In what follows we will often write $u$ instead of $u_{\Gamma}$. 
\begin{construction}\rm
\llabel{2015.07.17.constr1}
First we construct by induction on $n$ pairs $(H_n, \psi_n)$ where 
$$H_n:Ob_n(CC)\sr Ob_n({\cal C},p)$$
is a function and $\psi_n$ is a family of isomorphisms of the form
$$\psi_n(\Gamma):int(H(\Gamma))\sr I(\Gamma)$$
given for all $\Gamma\in Ob_n(CC)$ as follows (we will sometimes write $\psi$ instead of $\psi_n$ and $H$ instead of $H_n$):
\begin{enumerate}
\item For $n=0$ we set
$$H(pt)=pt$$
$$\psi(pt)=(\pi_{I(pt)})^{-1}:pt\sr I(pt)$$
\item For the successor of $n$, $\Gamma$ such that $H_n(\Gamma)=B$ and $\Delta\in Ob_1(\Gamma)$ we set
\begin{eq}\llabel{2015.07.19.eq4}
H_{n+1}(\Delta)=(B,\psi(Gamma)\circ u(\Delta))
\end{eq}
and
\begin{eq}\llabel{2015.07.19.eq5}
\psi(\Delta)=Q(\psi(\Gamma),u(\Delta))\circ \gamma_{\Delta}
\end{eq}
The fact that $\psi(\Delta)$ is an isomorphism follows from the inductive assumption, the assumption that $\gamma_{\Delta}$ is an isomorphism and Lemma \ref{2015.07.19.l1}.
\end{enumerate}
The functions $H_n$ define a function 
$$H_{Ob}:Ob(CC)\sr Ob(CC({\cal C},p))$$
where $H_{Ob}(\Gamma)=(l(\Gamma), H_l(\Gamma))$ that commutes with the length functions and functions $ft$.

For $f:\Gamma'\sr \Gamma$ define 
$$H_{Mor}(f)=(H_{Ob}(\Gamma'),(H_{Ob}(\Gamma),\psi(\Gamma')\circ I(f)\circ \psi(\Gamma)^{-1}))$$
This gives us a function 
$$H_{Mor}:Mor(CC)\sr Mor(CC({\cal C},p))$$
Note that we can also define $H_{Mor}(f)$ as the unique morphism such that
\begin{eq}\llabel{2015.07.19.eq1}
int(H_{Mor}(f))=\psi(\Gamma')\circ I(f)\circ \psi(\Gamma)^{-1}
\end{eq}
Without using any more assumptions on $I$, $\gamma$ and $u$ one verifies easily that the pair $H=(H_{Ob},H_{Mor})$ is a functor from the underlying category of $CC$ to the underlying category of $CC({\cal C},p)$.

In view of Lemma \ref{2015.06.29.l1} it remains to verify that $H$ satisfies the $p$-morphism and the $q$-morphism conditions of Definition \ref{2015.06.29.def1}. 

For the $p$-condition we need to verify that $H(p_{\Gamma})=p_{H(\Gamma)}$ for all $\Gamma$. Since both sides have the same domain and codomain and $int$ is bijective on morphisms with the a given domain and codomain it is sufficient to verify that
$$int(H(p_{\Gamma}))=int(p_{H(\Gamma)})$$
We proceed by induction on $n=l(\Gamma)$:
\begin{enumerate}
\item for $n=0$
$$int(H(Id_{pt}))=int(Id_{pt})=int(p_{pt})=int(p_{H(pt)})$$
\item for the successor of $n$ let $\Delta\in Ob_{n+1}(CC)$ and $\Gamma=ft(\Delta)$.  Then $\Delta\in Ob_1(\Gamma)$ and 
$$int(H(p_{\Delta}))=\psi(\Delta)\circ I(p_{\Delta})\circ \psi(\Gamma)^{-1}=Q(\psi(\Gamma),u(\Delta))\circ \gamma_{\Delta}\circ I(p_{\Delta})\circ \psi(\Gamma)^{-1}=$$$$Q(\psi(\Gamma),u(\Delta))\circ p_{I(\Gamma),u(\Delta)}\circ \psi(\Gamma)^{-1}$$
where the first equality is by (\ref{2015.07.19.eq1}), the second one by (\ref{2015.07.19.eq5}) and the third one by condition (3) of the problem. On the other hand we have
$$int(p_{H(\Delta)})=int(p_{(n+1,(H_n(\Gamma),\psi(\Gamma)\circ u(\Delta)))})=p_{int(H(\Gamma)),\psi(\Gamma)\circ u(\Delta)}$$
where the first equality is by (\ref{2015.07.19.eq4}) and the second by (\ref{2015.07.19.eq7}). 
Composing with $\psi(\Gamma)$ we get
$$int(H(p_{\Delta}))\circ \psi(\Gamma)=Q(\psi(\Gamma),u(\Delta))\circ p_{I(\Gamma),u(\Delta)}$$
$$int(p_{H(\Delta)})\circ \psi(\Gamma)=p_{int(H(\Gamma)),\psi(\Gamma)\circ u(\Delta)}\circ \psi(\Gamma)$$
and these expressions are equal by commutativity of the squares (\ref{2015.07.11.eq1}). 
\end{enumerate}
To prove the $q$-condition let us verify first that for $f:\Gamma'\sr \Gamma$ and $\Delta\in Ob_1(\Gamma)$ one has
\begin{eq}\llabel{2015.07.19.eq8}
H(f^*(\Delta))=H(f)^*(H(\Delta))
\end{eq}
Let $H(\Gamma')=(m,A)$ and $H(\Gamma)=(n,B)$. Then 
$$H(f^*(\Delta))=(m+1,(A,\psi(\Gamma')\circ u(f^*(\Delta))))$$
by (\ref{2015.07.19.eq4}) and 
$$H(f)^*(H(\Delta))=H(f)^*(n+1,(B,\psi(\Gamma)\circ u(\Delta)))=(m+1,(A,int(H(f))\circ\psi(\Gamma)\circ u(\Delta)))$$
where the first equality holds by (\ref{2015.07.19.eq4}) and the second by (\ref{2015.07.19.eq2}). Next one has 
$$\psi(\Gamma')\circ u(f^*(\Delta))=\psi(\Gamma')\circ I(f)\circ u(\Delta)$$
by condition (2) of the problem and 
$$int(H(f))\circ\psi(\Gamma)\circ u(\Delta)=\psi(\Gamma')\circ I(f)\circ \psi(\Gamma)^{-1}\circ \psi(\Gamma)\circ u(\Delta)=\psi(\Gamma')\circ I(f)\circ u(\Delta)$$
by (\ref{2015.07.19.eq1}). 

The equality (\ref{2015.07.19.eq8}) implies that the morphisms $H(q(f,\Delta))$ and $q(H(f),H(\Delta))$ have the same domain and codomain. Therefore to prove that they are equal it is sufficient to prove that they become equal after application of $int$. We further compose both sides with $\psi(\Delta)$. Then we have
$$int(H(q(f,\Delta)))\circ \psi(\Delta)=\psi(f^*(\Delta))\circ I(q(f,\Delta))=Q(\psi(\Gamma'),I(f)\circ u(\Delta))\circ \gamma_{f^*(\Delta)}\circ I(q(f,\Delta))=$$$$Q(\psi(\Gamma'),I(f)\circ u(\Delta))\circ Q(I(f),u(\Delta))\circ \gamma_{\Delta}=Q(\psi(\Gamma')\circ I(f), u(\Delta))\circ \gamma_{\Delta}$$
where the first equality is by (\ref{2015.07.19.eq1}), the second by (\ref{2015.07.19.eq5}), the third by condition (4) of the problem and the fourth by Lemma \ref{2015.04.14.l0}. On the other hand
$$int(q(H(f),H(\Delta)))\circ \psi(\Delta)=int(q(H(f),(n+1,(B,\psi(\Gamma)\circ u(\Delta))))) \circ \psi(\Delta)=$$$$Q(int(H(f)),\psi(\Gamma)\circ u(\Delta))\circ \psi(\Delta)=
Q(int(H(f)),\psi(\Gamma)\circ u(\Delta))\circ Q(\psi(\Gamma'),u(\Delta))\circ \gamma_{\Delta}=$$$$Q(int(H(f))\circ \psi(\Gamma), u(\Delta))\circ \gamma_{\Delta}=Q(\psi(\Gamma')\circ I(f), u(\Delta))\circ \gamma_{\Delta}$$
where the first equality is by (\ref{2015.07.19.eq4}), the second by (\ref{2015.07.19.eq3}), the third by (\ref{2015.07.19.eq5}), the fourth by Lemma \ref{2015.04.14.l0} and the fifth by  (\ref{2015.07.19.eq1}). This completes Construction \ref{2015.07.17.constr1}.
\end{construction}
\begin{remark}\rm
Homomorphisms $H(\Phi,\phi,\wt{\phi})$ can be obtained as particular cases of homomorphisms $H(I,u,\gamma)$. More precisely, we can state without a proof that
$$H(\Phi,\phi,\wt{\phi})=H(I,u,\gamma)$$
where:
\begin{enumerate}
\item $I(\Gamma)=\Phi(int(\Gamma))$ and $I(f)=\Phi(int(f))$,
\item for $\Gamma=(n,B)$ and $\Delta=(n+1,(B,F))$, 
$$u_{\Gamma}(\Delta)=\Phi(F)\circ \phi$$
\item for $\Gamma=(n,B)$ and $\Delta=(n+1,(B,F))$, $\gamma_{\Delta}$ is the ``natural'' isomorphism from $(\Phi(int(B));\Phi(F)\circ \phi)$ to $\Phi(int(B);F)$. More precisely 
$$\gamma_{\Delta}=((p_{\Phi(int(B)),\Phi(F)\circ \phi})*(Q(\Phi(F)\circ \wt{\phi})))^{-1}$$
\end{enumerate}
\end{remark}
\begin{lemma}
\llabel{2015.07.19.l2}
Let $I$, $u$ and $\gamma$ be as in Problem \ref{2015.07.17.prob1} and let $H$ be the corresponding homomorphism of Construction \ref{2015.07.17.constr1}. Then one has:
\begin{enumerate}
\item If $I$ is a faithful functor and $u_{\Gamma}$ are injective then $H$ is an injection of C-systems.
\item If $I$ is a fully faithful functor and $u_{\Gamma}$ are bijective then $H$ is an isomorphism of C-systems.
\end{enumerate}
\end{lemma}
\begin{proof}
Both statements in relation to objects have straightforward proofs by induction on the length. In relation to morphisms the statements follow from the ones about the objects, the fact that $int$ is fully faithful and formula (\ref{2015.07.19.eq1}).
\end{proof}

\subsection{Every C-system is isomorphic to a C-system of the form $CC({\cal C},p)$}

\begin{problem}
\llabel{2014.09.18.prob3}
Let $CC$ be a C-system. Construct a universe category $({\cal C},p)$ and an isomorphism $CC\cong CC({\cal C},p)$.
\end{problem}

We will provide three different constructions for this problem - Constructions \ref{2014.09.18.constr3}, \ref{2015.07.21.constr1} and \ref{2015.07.25.constr1} with the two latter constructions using the first one.

It is customary in the modern mathematics to use ``the'' category of sets $Sets$. In fact, every set $X$ in the Zermelo-Fraenkel theory defines a category $S(X)$ where:
$$Ob(S(X))=X$$
$$Mor(S(X))=\amalg_{x_1,x_2\in X}Fun(x_1,x_2)$$
where $Fun(x_1,x_2)$ is the set of functions from $x_1$ to $x_2$. This definition makes sense since elements of Zermelo-Fraenkel sets are themselves Zermelo-Fraenkel sets. 

Taking $X$ to be sets satisfying particular conditions, e.g. Grothendieck universes, one obtains categories that can be equipped with various familiar structures such as fiber products, internal $Hom$-objects etc. When one says consider ``the'' category of sets one presumably means the category $S(GU)$ for a chosen Grothendieck universe $GU$.

The first construction that we provide assumes that we are working in set theory with a chosen Grothendieck universe (or in type theory with a chosen type theoretic universe) that contains the sets of objects and morphisms of our C-system. In the case of a type theory we will actually need two universes in order to have a type of which the first universe is an object. 

In what follows we use the notations
$$\wt{Ob}(CC)=\{s\in Mor(CC)\,|\,s:ft(X)\sr X, l(X)>0, s\circ p_X=Id_{X}\}$$
and $\partial:\wt{Ob}(CC)\sr Ob(CC)$, $\partial(s)=codom(s)$ that were introduced in \cite{Csubsystems}. We may sometimes abbreviate $Ob(CC)$ to $CC$. 

We will write $\wt{Ob}_1(\Gamma)$ for the subset of $\wt{Ob}$ that consists of $s$ such that $ft(\partial(s))=\Gamma$. For $f:\Gamma'\sr \Gamma$ we have the function $s\mapsto f^*(s,1)$ (see \cite{Csubsystems}) that maps $s$ to the element of $\wt{Ob}_1(\Gamma)$ that is the pull-back of the section $s$ relative to $f$. We will denote this function by $f^*$. It is easy to verify from the definitions that 
\begin{eq}\llabel{2015.07.19.eq10}
(Id_{\Gamma})^*(s)=s
\end{eq}
and for $g:\Gamma''\sr \Gamma'$, $f:\Gamma'\sr \Gamma$ and $s\in \wt{Ob}_1(\Gamma)$ one has 
\begin{eq}\llabel{2015.07.19.eq11a}
g^*(f^*(s))=(g\circ f)^*(s).
\end{eq}
i.e., that the maps $f^*$ define on the family of sets $\wt{Ob}_1$ the structure of a presheaf. We continue using the notation $\wt{Ob}_1$ for this presheaf. 
\begin{construction}
\llabel{2014.09.18.constr3}\rm
Denote by $PreShv(CC)$ the category of presheaves on the precategory underlying $CC$, i.e., the category of contravariant functors from the precategory underlying $CC$ to $Sets$. 

Let $Ob_1$ be the presheaf that takes an object $\Gamma\in CC$ to the set $Ob_1(\Gamma)$
and a morphism $f:\Gamma'\sr \Gamma$ to the map $\Delta\mapsto f^*(\Delta)$. It is a functor due to the composition and unity axioms for $f^*$. 

Let $\wt{Ob}_1$ be the presheaf that takes $\Gamma$ to $\wt{Ob}_1(\Gamma)$ described above. 

Let further $\partial:\wt{Ob}_1\sr Ob_1$ be the morphism that takes $s$ to $\partial(s)$. It is well defined as a morphisms of families of sets and forms a morphism of presheaves since $\partial(f^*(s))=f^*(\partial(s))$. 

The morphism $\partial$ carries a universe structure that is defined by the standard pull-back squares in the category of presheaves. 

We are going to construct a homomorphism $CC\sr CC(PreShv(CC),\partial)$ using Construction \ref{2015.07.17.constr1} and to show that it is an isomorphism using Lemma \ref{2015.07.19.l2}. 

We set $Yo$ to be the Yoneda embedding. 

We set 
$$v_{\Gamma}:Ob_1(\Gamma)\sr Hom_{PreShv}(Yo(\Gamma),Ob_1)$$
to be the standard bijections between sections of the presheaf $Ob_1$ on an object $\Gamma$ and morphisms from the corresponding representable presheaf $Yo(\Gamma)$ to $Ob_1$ in the category of presheaves.  It follows easily from the definitions that for $f:\Gamma'\sr \Gamma$ and $\Delta\in Ob_1(\Gamma)$ one has
\begin{eq}\llabel{2015.07.19.eq17}
v_{\Gamma'}(f^*(\Delta))=Yo(f)\circ v_{\Gamma}(\Delta)
\end{eq}
We also set
$$\wt{v}_{\Gamma}:\wt{Ob}_1(\Gamma)\sr Hom_{PreShv}(Yo(\Gamma),\wt{Ob}_1)$$
to be the bijections of the same form for $\wt{Ob}_1$. Again, it follows easily from the definitions that for $f:\Gamma'\sr \Gamma$ and $s\in \wt{Ob}_1(\Gamma)$ one has
\begin{eq}\llabel{2015.07.19.eq16}
\wt{v}_{\Gamma'}(f^*(s))=Yo(f)\circ \wt{v}_{\Gamma}(s)
\end{eq}

To construct $\gamma$ we first need to prove a lemma. Recall that for $\Delta\in CC$ such that $l(\Delta)>0$ we let $\delta(\Delta): \Delta \sr p_{\Delta}^*(\Delta)$ denote the section of $p_{p_{\Delta}^*(\Delta)}$ given by the diagonal. We have 
$$\delta(\Delta)\in \wt{Ob}_1(\Delta)$$
\begin{lemma}
\llabel{2009.12.28.l1}
Let $\Gamma\in Ob(CC)$ and $\Delta\in Ob_1(\Gamma)$. Then the square
\begin{eq}
\llabel{2009.12.28.eq2}
\begin{CD}
Yo(\Delta) @>\wt{v}(\delta(\Delta))>> \wt{Ob}_1\\
@VYo(p_{\Delta})VV @VV\partial V\\
Yo(\Gamma) @>v(\Delta)>> Ob_1
\end{CD}
\end{eq}
is a pull-back square.
\end{lemma}
\begin{proof}
We have to show that for any $\Gamma'\in CC$ the function
\begin{eq}\llabel{2015.07.19.eq9}
Hom(\Gamma',\Delta)\sr Hom(\Gamma',\Gamma)\times_{Ob_1(\Gamma')}\wt{Ob}_1(\Gamma')
\end{eq}
defined by the square (\ref{2009.12.28.eq2}) is a bijection. Unfolding the definitions we see that this function sends $g:\Gamma'\sr \Delta$ to the pair $(g\circ p_{\Delta}, g^*(\delta(\Delta)))$ and that the fiber product is relative to the function from $Hom(\Gamma',\Gamma)$ to $Ob_1(\Gamma')$ that sends $f$ to $f^*(\Delta)$ and the function from $\wt{Ob}_1(\Gamma')$ to $Ob_1(\Gamma')$ that sends $s$ to $\partial(s)$.

Note that $g^*(\delta(\Delta))=s_g$ where $s$ is the $s$-operation of C-systems (see \cite[Definition 2.3]{Csubsystems}) and $g\circ p_{\Delta}$ is the morphism that we denoted in \cite{Csubsystems} by $ft(g)$. 

Let $f_1,f_2:\Gamma'\sr \Delta$ be two morphisms such that their images under  (\ref{2015.07.19.eq9}) coincide i.e. such that $ft(f_1)=ft(f_2)$ and $s_{f_1}=s_{f_2}$. This implies that $f_1=f_2$ in view of \cite[Definition 2.3(3)]{Csubsystems}. Therefore the function (\ref{2015.07.19.eq9}) is injective.

Let $f:\Gamma'\sr \Gamma$ be a morphism and $s\in \wt{Ob}_1(\Gamma')$ a section such that $\partial(s)=f^*(\Delta)$. Then the composition $s\circ q(f,\Delta)$ is a morphism $f':\Gamma'\sr \Delta$ such that $f'\circ p_{\Delta}=f$. We also have 
$$s_{f'}=s_{s\circ q(f,\Delta)}=s_s=s$$
which proves that (\ref{2009.12.28.eq2}) is surjective. This completes the proof of Lemma \ref{2009.12.28.l1}.
\end{proof}
Let $\Gamma\in Ob(CC)$ and $\Delta\in Ob_1(\Gamma)$. By construction, $(Yo(\Gamma);v(\Delta))$ is the standard fiber product of the morphisms $v(\Delta)$ and $\partial$ in the category of presheaves. On the other hand $Yo(\Delta)$ is a fiber product of the same two morphisms by Lemma \ref{2009.12.28.l1}. Therefore there exists a unique isomorphism
$$\gamma_{\Delta}:(Yo(\Gamma);v(\Delta))\sr Yo(\Delta)$$
such that 
\begin{eq}\llabel{2015.07.19.eq14}
\gamma_{\Delta}\circ \wt{v}(\delta(\Delta))=Q(v(\Delta))
\end{eq}
and 
\begin{eq}\llabel{2015.07.19.eq15}
\gamma_{\Delta}\circ Yo(p_{\Delta})=p_{Yo(\Gamma),v(\Delta)}
\end{eq}

It remains to verify the four conditions of Problem \ref{2015.07.17.prob1} since the conditions of Lemma \ref{2015.07.19.l2}(2) are obviously satisfied. 

We have that $Yo(pt)\sr pt$ is an isomorphism. 

The second condition is (\ref{2015.07.19.eq17}). 

The third condition is (\ref{2015.07.19.eq15}). 

It remains to verify the fourth condition. Let $f:\Gamma'\sr \Gamma$ and $\Delta\in Ob_1(\Gamma)$. We need to show that
\begin{eq}\llabel{2015.07.19.eq11}
\gamma_{f^*(\Delta)}\circ Yo(q(f,\Delta))=Q(Yo(f),v(\Delta))\circ \gamma_{\Delta}
\end{eq}
Two of the morphisms that are involved in the condition can be seen on the diagram
$$
\begin{CD}
(Yo(\Gamma'); v(f^*(\Delta))) @>\gamma_{f^*(\Delta)}>> Yo(f^*(\Delta)) @>Yo(q(f,\Delta))>> Yo(\Delta) @>\wt{v}(\delta(\Delta))>> \wt{Ob}_1\\ 
@Vp_{Yo(\Gamma'), v(f^*(\Delta))}VV @VYo(p_{f^*(\Delta)})VV @VVYo(p_{\Delta})V @VVpV\\
Yo(\Gamma') @= Yo(\Gamma') @>Yo(f)>> Yo(\Gamma) @>v(\Delta)>> Ob_1
\end{CD}
$$
By Lemma \ref{2009.12.28.l1}, $Yo(\Delta)$ is a fiber product with the projections $Yo(p_{\Delta})$ and $\wt{v}(\delta(\Delta))$. Therefore it is sufficient to verify that the compositions of the two sides of (\ref{2015.07.19.eq11}) with the projections are equal, i.e., we have to prove two equalities:
\begin{eq}\llabel{2015.07.19.eq12}
\gamma_{f^*(\Delta)}\circ Yo(q(f,\Delta))\circ Yo(p_{\Delta})=Q(Yo(f),v(\Delta))\circ \gamma_{\Delta}\circ Yo(p_{\Delta})
\end{eq}
and
\begin{eq}\llabel{2015.07.19.eq13}
\gamma_{f^*(\Delta)}\circ Yo(q(f,\Delta))\circ \wt{v}(\delta(\Delta))=Q(Yo(f),v(\Delta))\circ \gamma_{\Delta}\circ \wt{v}(\delta(\Delta))
\end{eq}
For the (\ref{2015.07.19.eq12}) we have
$$\gamma_{f^*(\Delta)}\circ Yo(q(f,\Delta))\circ Yo(p_{\Delta})=\gamma_{f^*(\Delta)}\circ Yo(p_{f^*(\Delta)})\circ Yo(f)=$$$$p_{Yo(\Gamma'), v(f^*(\Delta))}\circ Yo(f)=p_{Yo(\Gamma'), Yo(f)\circ v(\Delta)}\circ Yo(f)$$
where the first equality is by the commutativity of the canonical squares in $CC$ and the fact that $Yo$ is a functor, the second by (\ref{2015.07.19.eq14}) and the third one by (\ref{2015.07.19.eq17}). On the other hand
$$Q(Yo(f),v(\Delta))\circ \gamma_{\Delta}\circ Yo(p_{\Delta})=Q(Yo(f),v(\Delta))\circ p_{Yo(\Gamma),v(\Delta)}=p_{Yo(\Gamma'), Yo(f)\circ v(\Delta)}\circ Yo(f)$$
where the first equality is by (\ref{2015.07.19.eq14}) and the second by the commutativity of the squares (\ref{2015.07.11.eq1}). 

For (\ref{2015.07.19.eq13}) we have 
$$\gamma_{f^*(\Delta)}\circ Yo(q(f,\Delta))\circ \wt{v}(\delta(\Delta))=\gamma_{f^*(\Delta)}\circ \wt{v}(q(f,\Delta)^*(\delta(\Delta)))=$$$$\gamma_{f^*(\Delta)}\circ \wt{v}(\delta(f^*(\Delta)))=Q(v(f^*(\Delta)))$$
where the first equality is by (\ref{2015.07.19.eq16}), the second follows from a simple computation in $CC$ and the third one is by (\ref{2015.07.19.eq14}). On the other hand one has
$$Q(Yo(f),v(\Delta))\circ \gamma_{\Delta}\circ \wt{v}(\delta(\Delta))=Q(Yo(f),v(\Delta))\circ Q(v(\Delta))=Q(Yo(f)\circ v(\Delta))=Q(f^*(\Delta))$$

where the first equality is by (\ref{2015.07.19.eq14}), the second one by Lemma \ref{2015.04.14.l0} and the third one by (\ref{2015.07.19.eq17}). 

This completes Construction \ref{2014.09.18.constr3}. 
\end{construction}

The second construction that we provide for Problem \ref{2014.09.18.prob3} is as follows. For a set $M$ let $Rp_n(M)$ be the set of subsets of $(\dots(M\times M)\times\dots)\times M$ where $M$ occurs in the expression $n+1$ times. Let $Rp(M)$ be the category with
$$Ob(Rp(M))=\amalg_{n\ge 0} Rp_n(M)$$
$$Mor(Rp(M))=\amalg_{(m,X),(n,Y)}Fun(X,Y)$$
where $Fun(X,Y)$ is the set of functions from $X$ to $Y$ and the identity morphisms and compositions of morphisms are given in the obvious way. 

If our theory has a universe then the category $Sets$ is defined and there is a functor $FF:Rp(M)\sr Sets$ that sends $(n,X)$ to $X$ and acts on morphisms in the obvious way. This functor is fully faithful. (Note that it is not an inclusion of categories since, for example, the empty subset is an element of each of $Rp_n(M)$ so that there are objects $(n,\emptyset)$ for all $n$ which are all mapped by $FF$ to the one empty set of $Sets$).

Let $CC$ be a C-system. Consider ${\cal C}=Funct(CC^{op},Rp(Mor(CC)))$. Define a universe $(p:\wt{O}\sr O, p_{\_}, Q(-))$ in $\cal C$ as follows. 

For $\Gamma\in CC$ let 
$$Ob'_1(\Gamma)=\{f\in Mor(CC)|\, f=Id_{codom(f)})$$
and let $\iota_{\Gamma}:Ob_1'(\Gamma)\sr Ob_1(\Gamma)$ be the bijection defined by the codomain function $codom$. These bijections together with the structure of a presheaf on the family of sets $Ob_1(\Gamma)$ define a structure of a presheaf on the family of sets $Ob_1'(\Gamma)$. Let 
$$O(\Gamma)=(0,Ob_1'(\Gamma))\in Rp(Mor(CC))$$
The structure of a presheaf of sets on $Ob_1'$ defines a structure of an element of $\cal C$ on $O$.

Next let
$$\wt{O}(\Gamma)=(0,\wt{Ob}_1(\Gamma))\in Rp(Mor(C))$$
The structure of a presheaf on $\wt{Ob}_1$ provide $\wt{O}$ with a structure of an object of $\cal C$. The morphism of presheaves $\partial$ defines in an obvious way a morphism $p:\wt{O}\sr O$ in $\cal C$. Let us construct a structure of a universe on $p$.

Let $X\in {\cal C}$ and $F:X\sr O$. For $\Gamma\in CC$ let $X(\Gamma)=(n,X_0(\Gamma))$ where $X_0(\Gamma)$ is an element of $Rp_n(Mor(CC))$. Then 
$$(X;F)_0(\Gamma)=\{(x_0,s)\in X_0(\Gamma)\times \wt{Ob}_1(\Gamma)\,|\, F(x_0)=Id_{codom(s)}\}\in Rp_{n+1}(Mor(CC))$$
is an element of $Rp_{n+1}(Mor(CC))$ and $(X;F)(\Gamma)=(n+1,(X;F)_0(\Gamma))$ is an element of $Rp(Mor(CC))$. 

It is easy to equip the family $(X;F)(\Gamma)$ of elements of $Rp(Mor(CC))$ with a structure of an object of $\cal C$ and equally easy to define morphisms $p_{X,F}:(X;F)\sr X$ and $Q(F): (X;F)\sr \wt{O}$. 

We have also a functor $I:CC\sr {\cal C}$ that extends the family of sets 
$$X\mapsto (Y\mapsto (0,Mor(X,Y)))$$
The image of the final object of $CC$ under this functor is a final object in $\cal C$ which completes the description of a universe category structure on $\cal C$.
\begin{lemma}
\llabel{2015.07.21.l1}
For any $F:X\sr O$ the square
\begin{eq}\llabel{2015.07.21.eq1}
\begin{CD}
(X;F) @>Q(F)>> \wt{O}\\
@Vp_{X,F}VV @VVpV\\
X @>F>> O
\end{CD}
\end{eq}
is a pull-back square in $\cal C$.
\end{lemma}
\begin{proof}
One can either give a direct  proof which would not require an extra universe or one can argue that the functor $FF$ defines a functor $\Phi:{\cal C}\sr PreShv(CC)$ which is fully faithful and which maps  squares (\ref{2015.07.21.eq1}) to standard pull-back squares in the category of presheaves of sets.
\end{proof}

\begin{problem}
\llabel{2015.07.21.prob1}
To construct an isomorphism $H:CC\sr CC({\cal C},p)$ where $({\cal C},p)$ is the universe category constructed above.
\end{problem}
There are two constructions for this problem. One we don't describe here because giving its detailed description would take a lot of space and add little understanding. It is a direct construction based on Construction \ref{2015.07.17.constr1} and Lemma \ref{2015.07.19.l2} that parallels Construction \ref{2014.09.18.constr3}. This direct construction would not use any extra universes and, in combination with the construction of $({\cal C},p)$ based on the direct proof of Lemma \ref{2015.07.21.l1} would provide a construction for Problem \ref{2014.09.18.prob3} that does not require any additional universes. 

The construction that we give below uses Construction \ref{2014.09.18.constr3} and therefore requires an extra universe.
\begin{construction}\rm
\llabel{2015.07.21.constr1}
Let $\Phi:{\cal C}\sr PreShv(CC)$ be the functor defined by $FF$. Since $FF$ is fully faithful so is $\Phi$. We have, by definition
$$\Phi(O)=Ob_1'$$
$$\Phi(\wt{O})=\wt{Ob}_1$$
The bijections $\iota_{\Gamma}$ give us an isomorphism of presheaves 
$$\iota:\Phi(O)=Ob_1'\sr Ob_1$$
which commute with $\Phi(p)$ and $\partial$ and together with $\Phi$ form a universe category functor $(\Phi,\iota,Id_{\wt{Ob}_1})$. This universe category functor satisfies the conditions of Lemma \ref{2014.09.18.l1}(2) and therefore the homomorphism $H(\Phi,\iota,Id_{\wt{Ob}_1})$ is an isomorphism. Composing the isomorphism of Construction \ref{2014.09.18.constr3} the inverse to this isomorphism we obtain a solution to Problem \ref{2015.07.21.prob1}.
\end{construction}
The direct construction of the universe category $({\cal C},p)$ does not increase the universe level but it uses the operation of taking the set of subsets that in type theory requires the propositional resizing rule in order to be defined inside a given universe. Here is an outline of a third construction that gives an even ``tighter'' universe category $({\cal C},p)$ with an isomorphism $CC\sr CC({\cal C},p)$.

\begin{construction}\rm
\llabel{2015.07.25.constr1}
Define by induction on $n$ pairs $(C_n,\Phi_n)$ where $C_n$ is a set and $\Phi_n:C_n\sr PreShc(CC)$ is a function as follows:
\begin{enumerate}
\item for $n=0$ we set $C_0=\{pt,U,\wt{U}\}$ and 
$$\Phi_0(pt)=pt$$
$$\Phi_0(U)=Ob_1$$
$$\Phi_0(\wt{U})=\wt{Ob}_1$$
where $pt$ on the right hand side of the first equality is the final object of $PreShv(CC)$,
\item for the successor of $n$ we set
$$C_{n+1}=\amalg_{X\in C_n}Hom_{PreShv(CC)}(\Phi_n(X), Ob_1)$$
and
$$\Phi_{n+1}(X,F)=(\Phi_n(X);F)$$
where $(X;F)$ is defined using standard fiber products in $PreShv(CC)$. 
\end{enumerate}
We then define
$$Ob({\cal C})=\amalg_{n\ge 0}C_n$$
$$Mor({\cal C})=\amalg_{(m,X),(n,Y)\in Ob({\cal C})}Hom_{PreShv(CC)}(\Phi_m(X),\Phi_n(Y))$$
The composition and the identity morphisms are defined in such a way as to make the pair of maps 
$$\Phi_{Ob}=\amalg_n \Phi_n$$
$$\Phi_{Mor}=\amalg_{(m,X),(n,Y)\in Ob({\cal C})} i(\Phi_m(X),\Phi_n(Y))$$
where $i(F,G)$ is the inclusion of $Hom_{PreShv(CC)}(F,G)$ into $Mor(PreShv(CC))$, into a functor. This functor, which we denote by $\Phi$, is then fully faithful. 

One proves easily that $pt$ is a final object of $\cal C$. One defines the universe morphism in $\cal C$ as the morphism $p:\wt{U}\sr U$ that is mapped by $\Phi$ to $\partial$. Given $(m,X)\in Ob({\cal C})$ and a morphism $F:(m,X)\sr U$ one defines $((m,X);F)$ as $(m+1,(X,\Phi(F)))$. This object is a vertex of the square
\begin{eq}\llabel{2015.07.25.eq1}
\begin{CD}
((m,X);F) @>Q(F)>> \wt{U}\\
@Vp_{(m,X),F}VV @VVpV\\
(m,X) @>F>> U
\end{CD}
\end{eq}
that is defined by the condition that it is mapped by $\Phi$ to the square
\begin{eq}\llabel{2015.07.25.eq2}
\begin{CD}
(\Phi_m(X);\Phi(F)) @>Q(\Phi(F))>> \wt{Ob}_1\\
@Vp_{\Phi_m(X),\Phi(F)}VV @VVpV\\
\Phi_m(X) @>\Phi(F)>> \wt{Ob}_1
\end{CD}
\end{eq}
Since $\Phi$ is fully faithful and the square (\ref{2015.07.25.eq2}) is a pull-back square, the square (\ref{2015.07.25.eq1}) is a pull-back square. This provides us with a universe structure on $p$ and completes the construction of the universe category $({\cal C},p)$. 

The functor $\Phi$ together with two identity morphisms forms a universe category functor ${\bf \Phi}=(\Phi,Id_{Ob_1},Id_{\wt{Ob}_1})$ that satisfies the conditions of Lemma \ref{2015.07.19.l2}(2). Therefore $\bf\Phi$ defines an isomorphism
$$CC({\cal C},p)\sr CC(PreShv(CC),\partial)$$
composing the isomorphism of Construction \ref{2014.09.18.constr3} with the inverse to this isomorphism we obtain a solution to Problem \ref{2014.09.18.prob3}. This completes Construction \ref{2015.07.25.constr1}.
\end{construction}
\begin{remark}\rm
\llabel{2015.07.25.rem1}
The category $\cal C$ of Construction \ref{2015.07.25.constr1} has all of the structures of a C-system and these structures satisfy all of the required properties except for the property that $l^{-1}(0)=\{pt\}$. We would like to call such objects ``generalized C-systems''. They seem to appear also in other examples and may play an important role in the future.
\end{remark}

\subsection{A universe category defined by a precategory}

The following problem was inspired by a question from an anonymous referee of \cite{Csubsystems}. Here we have to use the word precategory as in the definition of a C-system since the construction for this problem is not invariant under equivalences. Let us recall the following definition that also introduces the notations to be used below.
\begin{definition}
\llabel{2015.04.22.def1}
A category with fiber products is a category together with, for all pairs of morphisms of the form $f:X\sr Z$, $g:Y\sr Z$, fiber squares
$$
\begin{CD}
(X,f)\times_Z (Y,g) @>pr^{(X,f),(Y,g)}_2>> Y\\
@Vpr^{(X,f),(Y,g)}_1 VV @VVg V\\
X @>f>> Z
\end{CD}
$$
We will often abbreviate these main notations in various ways. The morphism $pr_2\circ g=pr_1\circ f$ from $(X,f)\times(Y,g)$ to $Z$ is denoted by $f\dd g$.
\end{definition}
\begin{problem}
\llabel{2015.07.25.prob1}
Let $C$ be a precategory with a final object $pt$ and fiber products. To construct a C-system $CC$ and an equivalence of categories $J_*:CC\sr C$, $J^*:C\sr CC$.
\end{problem}
\begin{remark}\rm
Note that if we required an isomorphism $CC\sr C$ then the problem would have no solution since, for example, there is no C-system whose set of objects is the set with two elements. Indeed, one of these elements, let us denote it by $X$, will have to have length $n>0$. Then $l(p_X^*(X))=l(X)+1$. Therefore $p_X^*(X)\ne X$ and $p_X^*(X)\ne pt$ which contradicts the assumption that $CC$ has only two objects. 
\end{remark}
We start with a general construction that does not require $C$ to have fiber products or a final object. The parts of it that do not concern C-systems must have certainly be known for a long time but we do not know where it was originally introduced.

For a precategory $C$ let $U_C$ be the presheaf such that
$$U_C(X)=\{(f,g)\,\,\,where\,\,\,f:X\sr Y\,\,\,and\,\,\,g:Z\sr Y\}$$
and for $a:X'\sr X$, 
$$U_C(a)(f,g)=(a\circ f,g)$$
One proves easily that this presheaf data defines a presheaf.

Let $\wt{U}_C$ be the presheaf such that
$$\wt{U}_C(X)=\{(f',g)\,\,\,where\,\,\,f':X\sr Z\,\,\,and\,\,\,g:Z\sr Y\}$$ 
and for $a:X'\sr X$, 
$$\wt{U}_C(a)(f',g)=(a\circ f',g)$$
Again one proves easily that $\wt{U}_C$ this presheaf data defines a presheaf. 

Let $p_C:\wt{U}_C\sr U_C$ be the morphism given by
$$(p_C)_X(f',g)=(f'\circ g,g)$$
One proves easily that this family of maps of sets is a morphism of presheaves. 

As in Construction \ref{2014.09.18.constr3} let $Yo$ be the Yoneda embedding and let 
$$v_{X}:U_C(X)\sr Hom_{PreShv}(Yo(X),U_C)$$
$$\wt{v}_X:\wt{U}_C(X)\sr Hom_{PreShv}(Yo(X),\wt{U}_C)$$
be the standard bijections which we will often write as $v$ and $\wt{v}$. 
\begin{lemma}
\llabel{2015.07.25.l1}
For $(f,g)\in U_C(X)$, $(f',g')\in \wt{U}_C(X')$ and $u:X'\sr X$ the square
\begin{eq}\llabel{2015.07.25.eq3}
\begin{CD}
Yo(X') @>\wt{v}(f',g')>> \wt{U}_C\\
@VYo(u)VV @VVp_CV\\
Yo(X) @>v(f,g)>> U_C
\end{CD}
\end{eq}
commutes if and only if $g'=g$ and the square
\begin{eq}\llabel{2015.07.25.eq4}
\begin{CD}
X' @>f'>> Z\\
@VuVV @VVgV\\
X @>f>> Y
\end{CD}
\end{eq}
commutes. The square (\ref{2015.07.25.eq3}) is a pull-back square if and only if $g=g'$ and the square (\ref{2015.07.25.eq4}) is a pull-back square.
\end{lemma}
\begin{proof}
The assertion about commutativity is obvious. The proof of the assertion about being a pull-back square is as follows. The square (\ref{2015.07.25.l1}) is a pull-back square if and only of for all $X''$ the corresponding square of sections on $X''$ is a pull-back square of sets. This square of sections is of the form
\begin{eq}\llabel{2015.07.25.eq5}
\begin{CD}
Hom(X'',X') @>r_1>>  \{(f'_0:X''\sr Z_0,\,\,\,g_0:Z_0\sr Y_0)\}\\
@Vs_1VV @VVr_2V\\
Hom(X'', X) @>s_2>> \{(f_0:X''\sr Y_0,\,\,\,g_0:Z_0\sr Y_0)\}
\end{CD}
\end{eq}
where $r_1(a')=(a'\circ f',g')$, $r_2(f'_0,g'_0)=(f'_0\circ g'_0,g'_0)$, $s_1(a)=a\circ u$, $s_2(a)=(a\circ f,g)$.

To check that (\ref{2015.07.25.eq5}) is a pull-back square it is sufficient to check that for every $a\in Hom(X'',X)$ the map $s_1^{-1}(a)\sr r_2^{-1}(s_2(a))$ defined by $r_1$ is a bijection. We have 
$$s_1^{-1}(a)=\{a':X''\sr X'\,|\,a'\circ u=a\}$$
and
\begin{eq}\llabel{2015.07.25.eq6}
r_2^{-1}(s_2(a))=\{(v:X''\sr Z,g)\,|\,v\circ g=a\circ f\}
\end{eq}
and the map defined by $r_1$ maps $a'$ to $(a'\circ v,g)$. 

Applying the same reasoning to the condition that the square is (\ref{2015.07.25.eq4}) is pull-back we see that it is equivalent to the condition that for all $X''$ and all $a:X''\sr X$ the map
from the set
$$(-\circ u)^{-1}(a)=\{a':X''\sr X'\,|\,a'\circ u=a\}$$
to the set
\begin{eq}\llabel{2015.07.25.eq7}
(-\circ g)^{-1}(a\circ f)=\{v:X''\sr Z\,|\,v\circ g=a\circ f\}
\end{eq}
given by $a'\mapsto a'\circ f'$, is a bijection. Since $g$ in (\ref{2015.07.25.eq6}) the sets on the right hand sides of (\ref{2015.07.25.eq6}) and (\ref{2015.07.25.eq7}) are in the obvious bijection that is compatible with the functions from $\{a':X''\sr X'\,|\,a'\circ u=a\}$ and therefore these two conditions are equivalent. 
\end{proof}
Applying our main construction to $(PreShv(CC),p_C)$ we obtain, for any precategory $C$, a C-system $CC(C)=CC(PreShv(CC),p_C)$.  
\begin{problem}
\llabel{2015.07.25.prob3}
Let $C$ be a precategory with a final object $pt$. To construct a function $J^*_1:Ob(C)\sr Ob_1(PreShv(C),p_C)$ and a family of isomorphisms $j_X:Yo(X)\sr int(J^*(X))$. 
\end{problem}
\begin{construction}\rm
\llabel{2015.07.25.constr3}
Let $X\in Ob(C)$. The pull-back square
$$
\begin{CD}
X @>Id_X>> X\\
@V\pi_XVV @VV\pi_XV\\
pt @>Id_{pt}>> pt
\end{CD}
$$
defines by Lemma \ref{2015.07.25.l1} a pull-back square
\begin{eq}\llabel{2015.07.25.eq8}
\begin{CD}
Yo(X) @>\wt{v}(Id_X,\pi_X)>> \wt{U}_C\\
@VYo(\pi_X)VV @VVp_CV\\
Yo(pt) @>v(Id_{pt},\pi_X)>> U_C
\end{CD}
\end{eq}
Let $\psi:pt\sr Yo(pt)$ be the unique isomorphism. Set 
$$J^*(X)=(pt,\psi\circ v(Id_{pt},\pi_X))$$
Then 
$$\pi_{Yo(X)}\circ \psi=Yo(\pi_X)$$
and therefore $\pi_{Yo(X)}*\wt{v}(Id_X,\pi_X)$ is a well defined morphism from $Yo(X)$ to $J^*(X)$.  It is easy to prove now that since (\ref{2015.07.25.eq8}) is a pull-back square this morphism is an isomorphism.
\end{construction}
\begin{remark}\rm
\llabel{2015.07.25.rem2}
For $C$ with a final object $pt$, the set $Ob_1(PreShv(CC),p_C)$ is in a constructive bijection with the set of pairs $(f:pt\sr Y,\,\,\, g:Z\sr Y)$ which is given, in the notation of Construction \ref{2015.07.25.constr3}, by the map $(f,g)\mapsto \psi\circ v(f,g)$. After composition with this bijection the function $J^*$ takes $X$ to $(Id:pt\sr pt, \pi_X:X\sr pt)$. The function $(f,g)\mapsto dom(g)$ defines a one-sided inverse to $J^*_1$ so that $J^*_1$ is always a split monomorphism.
\end{remark}
\begin{problem}
\llabel{2015.07.25.prob4}
Suppose that $C$ is a category with a final object $pt$ and fiber products. To construct a function $J_*:Ob(CC(C))\sr Ob(C)$ and for every $\Gamma\in Ob(CC(C))$ an isomorphism $\sigma_{\Gamma}:Yo(J_*(\Gamma))\sr int(\Gamma)$.
\end{problem}
\begin{construction}\rm 
\llabel{2015.07.25.constr4}
We first construct by induction on $n$, pairs $(J_n,\sigma_{n})$ where  
$$J_n:Ob_n(PreShv,p_C)\sr Ob(C)$$
and $\sigma_n$ is a family of isomorphisms 
$$\sigma_{n}(A):Yo(J_n(A))\sr int_n(A)$$
given for all $A\in Ob_n(PreShv,p_C)$ as follows (we write $J$ instead of $J_n$ and $\sigma$ instead of $\sigma_n$):
\begin{enumerate}
\item For $n=0$ we set $J(A)=pt$ and $\sigma(A):Yo(pt)\sr pt$ is the unique isomorphism,
\item For the successor of $n$ we proceed as follows. Let $(B,F)\in Ob_{n+1}$ where $B\in Ob_n$ and $F:int(B)\sr U_C$. Then 
$$v^{-1}(\sigma_B\circ F)\in U_C(J(B))$$
is of the form
$$v^{-1}(\sigma_B\circ F)=(f:J(B)\sr Z,\,\,\,g:Y\sr Z)$$
Let
$$J(B,F)=(J(B),f)\times_Z (Y,g)$$
To define $\sigma(B,F)$ consider the diagram
$$
\begin{CD}
Yo(J(B,F)) @>\iota>> (Yo(J(B)),\sigma(B)\circ F) @>Q(\sigma(B),F)>> (int(B);F) @>Q(F)>> \wt{U}_C\\
@VYo(pr_1)VV @Vp_{Yo(J(B)),\sigma(B)\circ F}VV @Vp_{int(B),F}VV @VVp_CV\\
Yo(J(B)) @= Yo(J(B)) @>\sigma(B)>> int(B) @>F>> U_C
\end{CD}
$$
where $\iota$ is the morphism $Yo(pr_1)*\wt{v}(pr_2,g)$. This morphism is defined because
$$Yo(pr_1)\circ \sigma(B)\circ F=Yo(pr_1)\circ v(f,g)=(pr_1\circ f,g)$$
and
$$\wt{v}(pr_2,g)\circ p_C=(pr_2\circ g,g)$$
and $pr_1\circ f=pr_2\circ g$. By Lemma \ref{2015.07.25.l1} the square 
$$
\begin{CD}
Yo(J(B,F)) @>\wt{v}(pr_2,g)>> \wt{U}_C\\
@VYo(pr_1)VV @VVp_CV\\
Yo(J(B)) @>v(f,g)>> U_C
\end{CD}
$$
is a pull-back square which implies that $\iota$ is an isomorphism. We define 
$$\sigma(B,F)=\iota\circ Q(\sigma_B,F)$$
The morphism $Q(\sigma(B),F)$ is an isomorphism by Lemma \ref{2015.07.19.l1} and therefore $\sigma(B,F)$ is an isomorphism.
\end{enumerate}
We now define $J_*$ as the sum over $n\in\nn$ of $J_n$. This completes Construction \ref{2015.07.25.constr4}. 
\end{construction}
\begin{remark}\rm
\llabel{2015.07.25.rem3}
Using the bijection of Remark \ref{2015.07.25.rem2} we can look at the function from pairs $(f:pt\sr Z,g:Z\sr Y)$ to $Ob(C)$ corresponding to $J_1$. This function is given by 
$$(f,g)\mapsto (pt,f)\times_Z(Y,g)$$
When we compose it with $J^*$ and consider $J_*(J^*(X))$ we obtain $(pt,Id_{pt})\times_pt(X,\pi_X)$. Depending on the choice of the fiber product this element of $Ob(C)$ may be equal to $X$ or not but in any case there is a natural in $X$ isomorphism $X\sr J_*(J^*(X))$.
\end{remark}
We can now provide the following construction for Problem \ref{2015.07.25.prob1}.
\begin{construction}\rm
\llabel{2015.07.25.constr2}
Let $J^*:Ob(C)\sr Ob(CC(C))$ be the composition of the function $J^*_1$ of Construction \ref{2015.07.25.constr3} with the inclusion of $Ob_1(PreShv,p_C)$ into $Ob(CC(PreShv,p_C))$. We can extend it to a functor data setting:
$$J^*_{Mor}(f:X\sr Y)=int^{-1}_{J^*(X),J^*(Y)}(j_X^{-1}\circ Yo(f)\circ j_Y)$$
where $int_{\Gamma,\Gamma'}$ is the bijection
$$Hom_{CC(C)}(\Gamma,\Gamma')\sr Hom(int(\Gamma),int(\Gamma'))$$
defined by the functor $int$. It is easy to prove from definitions that it is a functor and, using the fact that both $Yo$ and $int$ are fully faithful, that $J^*$ is fully faithful. 

Similarly we can extend $J_*$ of Construction \ref{2015.07.25.constr4} to a functor data setting
$$(J_*)_{Mor}(f:\Gamma'\sr \Gamma)=Yo^{-1}_{J_*(\Gamma'),J_*(\Gamma)}(\sigma_{\Gamma'}\circ int(f)\circ \sigma_{\Gamma}^{-1})$$
where $Yo_{X,Y}$ is the bijection
$$Hom_C(X,Y)\sr Hom_{PreShv}(Yo(X),Yo(Y))$$
defined by the Yoneda embedding. Again it is easy to prove from definitions that this functor data is a functor and using the fact that both $int$ and $Yo$ are fully faithful that $J_*$ is fully faithful. 

After $J_*$ and $J^*$ have been extended to morphisms it makes sense to ask whether the families of isomorphisms  $j_X$ and $\sigma_{\Gamma}$ are natural in $X$ and $\Gamma$ respectively and one verifies easily that they indeed are.

Let $X\in Ob(C)$ then we have an isomorphism
$$
\begin{CD}
Yo(J_*(J^*(X))) @>\sigma_{J^*(X)}>> int(J^*(X)) @>j_X^{-1}>> Yo(X)
\end{CD}
$$
which is natural in $X$ and applying to it $Yo_{J_*(J^*(X)), X}^{-1}$ we get an isomorphism 
$$J_*(J^*(X))\sr X$$
which is again natural in $X$, i.e., we obtained a functor isomorphism $J_*\circ J^*\sr Id$.

Similarly, starting with, 
$$
\begin{CD}
int(\Gamma) @>\sigma^{-1}_{\Gamma}>>  Yo(J_*(\Gamma)) @>j_{J_*(\Gamma)}>> int(J^*(J_*(\Gamma)))
\end{CD}
$$
one obtains a functor isomorphism $Id\sr J^*\circ J_*$. This completes the Construction \ref{2015.07.25.constr2}. 
\end{construction} 
\begin{remark}\rm
\llabel{2015.07.25.rem4a}
It might be possible to provide a construction for Problem \ref{2015.07.25.prob1} that does not increase the universe level.
\end{remark}
\begin{remark}\rm
\llabel{2015.07.25.rem4}
The C-system $CC(C)$ does not require a choice of a final of object or fiber products in $C$ and in particular does not depend on such a choice. 

The functor $J^*:C\sr CC(C)$ requires a choice of a final object for its construction and depends on this choice. Let $B$ be the bijection 
$$Ob_1(PreShv(CC),p_C)\sr \{(f:pt_1\sr Z,\,\,\,g:Y\sr Z)\}$$
of Remark \ref{2015.07.25.rem2} defined by the choice of a final object $pt_1$. Let $J_{1,1}^*$ and $J_{1,2}^*$ be the functions $Ob(C)\sr Ob_1(PreShv,p_C)$ of Construction \ref{2015.07.25.prob3} defined by the choice of the final object $pt_1$ and a final object $pt_2$ respectively. Then one has
$$B(J_{1,1}^*(X))=(Id:pt_1\sr pt_1,\,\,\,\pi_{X,1}:X\sr pt_1)$$
and
$$B(J_{1,2}^*(X))=(a:pt_1\sr pt_2,\,\,\,\pi_{X,2}:X\sr pt_2)$$
where $a:pt_1\sr pt_2$ is the unique morphism. This shows that $J_{1,1}^*\ne J_{1,2}^*$ if $pt_1\ne pt_2$ and in particular that $J^*$ depends on the choice of the final object.

The fact that $J_*$ depends on the choice of fiber products is seen from the formula for $J_*(J^*(X))$ given in Remark \ref{2015.07.25.rem3}. 
\end{remark}

{\bf Conjecture}
\llabel{2009.12.27.prop1}
Let $\cal C$ be a category, $CC$ be a C-system and $M:CC\sr {\cal C}$ a functor such that $M(pt_{CC})$ is a final object of $\cal C$ and $M$ maps distinguished squares of $CC$ to pull-back squares of $\cal C$. Then there exists a universe $p_M:\wt{U}_M\sr U_M$ in $PreShv({\cal C})$ and a C-system homomorphism $M':CC\sr CC(PreShv({\cal C}), p_M)$ such that the square
$$
\begin{CD}
CC @>M>> {\cal C}\\
@VVM'V @VVV\\
CC(PreShv({\cal C}), p_M) @>int>> PreShv(C)
\end{CD}
$$
where the right hand side vertical arrow is the Yoneda embedding, commutes up to a functor isomorphism.

\def\cprime{$'$}


\begin{thebibliography}{1}

\bibitem{Aczel}
Peter Aczel.
\newblock A general church-rosser theorem.
\newblock {\em Unpublished manuscript}, 1978.

\bibitem{RezkCompletion}
Benedikt Ahrens, Krzysztof Kapulkin, and Michael Shulman.
\newblock Univalent categories and the {R}ezk completion.
\newblock {\em Math. Structures Comput. Sci.}, 25(5):1010--1039, 2015.

\bibitem{Cartmell0}
John Cartmell.
\newblock Generalised algebraic theories and contextual categories.
\newblock {\em Ph.D. Thesis, Oxford University}, 1978.
\newblock \url{https://uf-ias-2012.wikispaces.com/Semantics+of+type+theory}.

\bibitem{Cartmell1}
John Cartmell.
\newblock Generalised algebraic theories and contextual categories.
\newblock {\em Ann. Pure Appl. Logic}, 32(3):209--243, 1986.

\bibitem{FPT}
Marcelo Fiore, Gordon Plotkin, and Daniele Turi.
\newblock Abstract syntax and variable binding (extended abstract).
\newblock In {\em 14th {S}ymposium on {L}ogic in {C}omputer {S}cience
  ({T}rento, 1999)}, pages 193--202. IEEE Computer Soc., Los Alamitos, CA,
  1999.

\bibitem{HM2007}
Andr{\'e} Hirschowitz and Marco Maggesi.
\newblock Modules over monads and linearity.
\newblock In {\em Logic, language, information and computation}, volume 4576 of
  {\em Lecture Notes in Comput. Sci.}, pages 218--237. Springer, Berlin, 2007.

\bibitem{NTS}
Vladimir Voevodsky.
\newblock Notes on type systems.
\newblock {\em \url{https://github.com/vladimirias/old_notes_on_type_systems}},
  2009-2012.

\bibitem{Cofamodule}
Vladimir Voevodsky.
\newblock {C-system} of a module over a monad on sets.
\newblock {\em arXiv 1407.3394, submitted}, pages 1--20, 2014.

\bibitem{Csubsystems}
Vladimir Voevodsky.
\newblock Subsystems and regular quotients of {C-systems}.
\newblock In {\em Conference on Mathematics and its Applications, (Kuwait City,
  2014)}, number to appear, pages 1--11, 2015.

\end{thebibliography}

\comment{
\begin{definition}
\llabel{2009.12.27.def1}
Let $CC$ be a C-system. A universe model of $CC$ is a pair of a universe category $({\cal C},p)$ and a C-system homomorphism $CC\sr CC({\cal C},p)$.
\end{definition}

}

\comment{
???
\begin{remark}\rm
Construction \ref{} requires a inductive construction of families of sets that a priory can not be reduced to the inductive construction of families of subsets of a given set. 
In type theory this leads to the resulting object $({\cal C},p)$ being defined inside the same universe as $CC$ but the expression defining this object contains in it the next universe. 
\end{remark}
}

\end{document}